\newif\iftechreport
\techreporttrue  

\iftechreport
    \documentclass[11pt]{article}
\else
    \documentclass[opre,dblanonrev]{informs4}
\fi

\iftechreport

\usepackage{isetechreport}
\def\reportyear{25T}
\def\reportno{026}
\def\revisionno{0}
\def\originaldate{October 13, 2025}
\def\revisiondate{October 13, 2025}
\coraltrue
\cvcrfalse

\usepackage[english]{babel}
\usepackage[
  letterpaper,
  top=1in,
  bottom=1in,
  left=1in,
  right=1in
]{geometry}

\usepackage{amsmath,amssymb,amsfonts}
\usepackage{amsthm}
\usepackage[numbers]{natbib}
\usepackage{graphicx}
\usepackage[colorlinks=true, allcolors=blue]{hyperref}
\usepackage{optidef}
\usepackage{enumerate}
\usepackage[dvipsnames]{xcolor}
\usepackage{mathtools}\usepackage{algorithm,algpseudocode}
\usepackage{complexity}
\usepackage{subcaption}
\usepackage{tkz-graph}
\usepackage{tikz}

\newtheorem{theorem}{Theorem}

\newtheorem{definition}{Definition}
\newtheorem{lemma}{Lemma}
\newtheorem{corollary}{Corollary}
\newtheorem{remark}{Remark}

\newtheorem{observation}{Observation}
\newtheorem{proposition}{Proposition}
\newtheorem{example}{Example}

\newcommand{\N}{\mathbb{N}}
\renewcommand{\R}{\mathbb{R}}
\newcommand{\Z}{\mathbb{Z}}
\renewcommand{\H}{\mathcal{H}}
\renewcommand{\D}{\mathcal{D}}
\newcommand{\X}{\mathcal{X}}

\DeclareMathOperator*{\argmin}{arg\,min}
\newcommand{\NP}{\NPcomplexity}
\newcommand{\wce}{\textsc{WCE-FEAS}}
\newcommand{\sce}{\textsc{SCE-FEAS}}

\else

\usepackage{eqndefns-left} 
\RequirePackage{tgtermes}
\RequirePackage{newtxtext}
\RequirePackage{newtxmath}
\RequirePackage{bm}
\RequirePackage{endnotes}

\OneAndAHalfSpacedXII 

\usepackage{algorithm}
\usepackage{algpseudocode}
\usepackage{tikz}

\usepackage{amsfonts}
\usepackage[colorlinks=true, allcolors=blue]{hyperref}
\usepackage{enumerate}
\usepackage[dvipsnames]{xcolor}
\usepackage{mathtools}
\usepackage{algorithm,algpseudocode}
\usepackage{subcaption}

\newcommand{\NP}{\ensuremath{\mathbb{NP}}}
\newcommand{\N}{\mathbb{N}}
\newcommand{\R}{\mathbb{R}}
\newcommand{\Z}{\mathbb{Z}}
\renewcommand{\H}{\mathcal{H}}
\newcommand{\D}{\mathcal{D}}
\newcommand{\X}{\mathcal{X}}

\newcommand{\wce}{\textsc{WCE-FEAS}}
\newcommand{\sce}{\textsc{SCE-FEAS}}

\usepackage{natbib}
 \bibpunct[, ]{(}{)}{,}{a}{}{,}%
 \def\bibfont{\small}%
\EquationsNumberedThrough  

\TheoremsNumberedThrough
\ECRepeatTheorems  

\MANUSCRIPTNO{}
\fi

\begin{document}

\iftechreport

\title{Counterfactual Explanations for Integer Optimization Problems }
\author[1]{Felix Engelhardt}
\affil[1]{RWTH Aachen University}
\author[2]{Jannis Kurtz}
\author[2]{Ş. İlker Birbil}
\affil[2]{University of Amsterdam, Plantage Muidergracht 12, 1018 TV
  Amsterdam}
\author[3]{Ted Ralphs}
\affil[3]{Department of Industrial and Systems Engineering, Lehigh University,
Bethlehem, PA 18015}

\fi

\maketitle

\iftechreport

\begin{abstract}
Counterfactual explanations (CEs) offer a human-understandable way to explain decisions by identifying specific changes to the input parameters of a \emph{base} or \emph{present} model that would lead to a desired change in the outcome. For optimization models, CEs have primarily been studied in limited contexts and little research has been done on CEs for general integer optimization problems. In this work, we address this gap. We first show that the general problem of constructing a CE is $\Sigma_2^p$-complete even for binary integer programs with just a single \emph{mutable} constraint. Second, we propose solution algorithms for several of the most tractable special cases: (i) mutable objective parameters, (ii) a single mutable constraint, (iii) mutable right-hand-side, and (iv) all input parameters can be modified. We evaluate our approach using classical knapsack problem instances, focusing on cases with mutable constraint parameters. Our results show that our methods are capable of finding optimal CEs for small instances involving up to $40$ items within a few hours. Additionally, we present experiments on the resource constrained shortest path problem.

\vspace{0.2cm}
\noindent \textbf{Keywords}: counterfactual explanations, integer optimization, complexity, constraint generation
\end{abstract}
\fi

\section{Introduction}

Automated decision-making increasingly shapes critical aspects of modern society, ranging from applications in medical scheduling and financial planning to supply chain optimization and disaster management. Many of these applications are modeled as integer optimization problems (IPs) and are solved by state-of-the-art methods, such as branch and cut. Despite their widespread use, the reasoning behind optimization-based decisions is often unclear to those affected. For instance, optimization algorithms are used for automated planning of surgeries in hospitals \citep{OR_review_2011,ORHC_review_2025}. A patient denied an earlier surgery date may ask for an explanation for this decision. In humanitarian logistics, optimal robust depot locations are calculated to quickly dispatch relief items to a region affected by a disaster \citep{stienen2021optimal}. Here, the properties of an optimal solution significantly depend on the choice of optimization model and uncertainty set. Officials from a certain region may ask for an explanation why their region did not receive a depot while another region did.

One established way to provide explanations for decisions made by automated systems is based on the paradigm of counterfactual explanations (CEs). A CE is a minimal change in model inputs that would lead to a desired change in outcome. For example, in the aforementioned depot location problem, a CE could be: “If the population in your region were 5\% higher and the transportation time to region~B were 3\% smaller, it would be optimal to open a depot in your region.” Counterfactuals have become a cornerstone of explainable artificial intelligence that offers individuals understandable and actionable insights into automated predictions \citep{Verma2024,maragno2022,maragno2024finding}. Their role in optimization problems--particularly those involving integrality constraints--has only begun to be explored. First works studied the concept of counterfactuals for binary linear optimization problems where only objective parameters are mutable \citep{korikov2021counterfactual,korikov2021counterfactualGDPR,korikov2023objective}, drawing on techniques from constraint programming and inverse optimization. Later, \cite{kurtz2024} identified three types of counterfactuals of interest for optimization problems—weak, strong, and relative CEs—and analyzed them for linear optimization with continuous decision variables. In \cite{lefebvre2025computing} the authors derive a heuristic for calculating weak CEs based on a penalty alternating direction method.  Relative CEs have since been investigated for contextual stochastic optimization \citep{ramirez2025relative}, and CEs have been used to explain efficiency targets in data envelopment analysis \citep{bogetoft2024counterfactual} and to detect minimum number of constraints to make a problem infeasible in constraint programming \citep{gupta2024counterfactual}. Coming back to the case of integer optimization, however, the aforementioned three types of CEs remain largely unexplored, particularly when mutable parameters appear in the constraints.  

Constructing CEs is closely related to a number of existing areas of research in optimization. The most obvious connection is to global sensitivity analysis, but rather than quantifying how outcomes vary under systematic parameter changes, CEs ask how parameters must be minimally changed to induce specified outcome changes. This perspective highlights the relationship to inverse optimization---indeed, one may view CE construction as a generalized form of inverse optimization \citep{chan2021inverse,Wang2013}. It also highlights connections to bilevel optimization: the model structures induced by weak CEs are similar to those in optimistic bilevel formulations \citep{dempe2015bilevel,dempefoundations,kleinert2021survey}, whereas strong CEs naturally correspond to pessimistic bilevel formulations \citep{wiesemann2013pessimistic}. Certain special cases connect directly to well-studied problem classes. The case in which only the objective function can be modified subsumes the most classical version of inverse optimization, in which one seeks an objective that makes a given solution optimal. The case in which only the right-hand-side can vary relates both to multiobjective optimization via the restricted value function and to the standard form of mixed integer bilevel linear optimization \citep{dempefoundations}. Finally, the most general case considered here is equivalent to a bilevel optimization model in which the lower-level problem is a mixed integer bilinear optimization problem.

This paper fills a gap in the literature by studying weak and strong counterfactual explanations for integer linear optimization problems (ILPs) in which mutable parameters may appear both in the objective function and in the constraints. Our contributions are as follows:

\begin{itemize}
\item We present complexity results showing that constructing CEs in both the weak and strong cases is $\Sigma_2^p$-hard, even if the mutable parameters only appear in a single constraint.
\item We propose solution algorithms to calculate weak and strong CEs for the case when only a single constraint may contain mutable parameters. We distinguish four different cases: the mutable parameters appear (i) only in the objective function, (ii) only in the constraints, (iii) only on the right-hand side of the constraints, or (iv) everywhere in the problem.
\item We present computational experiments based on knapsack instances from \texttt{kplib}, showing that optimal CEs can be computed in a few hours for instances containing up to $40$ items. Additionally, we evaluate applications of CEs regarding resource constraint shortest path problems.  
\end{itemize}

\section{Definitions}
\label{sec:problem_definition}

In the following, we describe the basic definitions and the mathematical framework on which our study is based. To formally define the notion of a counterfactual explanation, we first need to define a so-called \emph{present optimization problem} on which the current decision, which we refer to as the \emph{present solution}, is based. In this paper, we consider the following form for the present problem:
\begin{equation}\label{eq:knapsack_problem}
    \begin{aligned}
    \min \ & \hat c^\top x \\
    s.t. \quad & \hat A x\ge \hat b, \\
    & x\in\X,
\end{aligned}
\end{equation}
where $\hat c\in \Z^n$ (the objective function), $\hat b\in\Z^n$ (the right-hand-side), and $\hat A\in\Z^{m\times n}$ (the constraint matrix) are the present values of the mutable parameters, i.e., those that can be modified to obtain a CE. The set $\X\subseteq \mathbb Z_+^n$ is the immutable feasible set, consisting of constraints that cannot be modified. In many applications, the set $\X$ is the feasible set of a MILP, though for the algorithms described later, we only require the set $\X$ be the feasible set of an optimization problem for which we have an appropriate black-box solver. 

To reason about counterfactuals, we introduce two additional sets. The \textit{favored solution space}, denoted by $\D\subseteq \mathbb Z_+^n$, consists of solutions that exhibit desirable properties and that we attempt to make optimal by modifying the mutable parameters to generate a CE. The \textit{mutable parameter space}, denoted by $H \subseteq \Z^n \times \Z^{m \times n} \times \Z^n$, are the triples $(c, A, b)$ of values that the mutable parameters are allowed to take on in attempting to generate a CE. We assume that $\D \cap \X \neq \emptyset$ and that $\H$ is bounded and contains $(\hat c,\hat A, \hat b)\in \H$. 

The structure of $\H$ depends on the nature of the application, but the following are the most typical ways in which the mutable parameter space would be specified:
\begin{enumerate}
    \item Only objective function parameters are mutable: $\H = \H_c \times \{ \hat A\} \times \{ \hat b\}$
    \item Only constraint parameters are mutable: $\H = \mathcal \{\hat c\} \times \H_A \times  \H_b $ 
    \item Only the right-hand side parameters are mutable: $\H = \mathcal \{\hat c\} \times \{\hat\H_A\} \times  \H_b $ 
    \item All parameters are mutable: $\H = \H_c \times \H_A \times \H_b $ 
\end{enumerate}
In principle, $\H$ can be any (MILP-representable) set. In the remainder of the paper, we choose $\H$ to be a discrete set for two reasons. First, in most applications, parameter values are naturally discrete. Even if we allow the parameters to take on rational values, we can still scale them to be integral without loss of generality. Second, allowing parameter values to be continuous can result in feasible sets that are open, which leads to difficulties in constructing CEs; see \cite{kurtz2024} for a discussion on open feasible regions. 

We now define the problem of constructing weak and strong CEs. 
Given a present problem of the form~\eqref{eq:knapsack_problem}, a favored solution space $\D$, and a mutable parameter space $\H$, a \emph{weak CE} is any $(c,A,b) \in \H$ such that \emph{some} optimal solution of the resulting modified instance lies in $\D$, i.e., fulfills the desired properties. 
\begin{definition}\label{def:weak_CE}
    A weak counterfactual explanation for \eqref{eq:knapsack_problem} is a point $(c,A,b)\in \H$ that satisfies the weak CE condition:
    \begin{equation}\label{eq:weak_CE_definition}
        \begin{aligned}
             \D \cap \argmin_{x\in \X: A x\ge b} c^\top x \neq \emptyset.
        \end{aligned}
    \end{equation}
\end{definition}
\noindent In general, we are interested in CEs that modify the present problem as little as possible. To measure the amount of modification, we utilize a distance measure $\delta: \H \times \H \to \R_{\ge 0}$ and define the \emph{cost} of a weak CE to be $\delta\left( (c,A,b), (\hat c, \hat A, \hat b)\right)$. We call a weak CE with minimal cost an \emph{optimal weak CE}. An optimal weak CE can be computed by solving the following optimistic bilevel problem.
\begin{equation}\label{eq:weak_CE_problem}
    \begin{aligned}
            \min_{c,A,b,x} \ & \delta\left( (c,A,b), (\hat c, \hat A, \hat b)\right) \\
            s.t.\quad & x\in \argmin_{x\in \X: A x\ge b} c^\top x ,\\
            & x\in \D, \\
            & (c,A,b) \in \H.
        \end{aligned}
\end{equation}

When a problem has multiple optimal solutions, either due to its inherent structure or the characteristics of the solver used, a decision maker may want to ensure that \emph{all} optimal solutions satisfy the desired properties specified by $\D$. This motivates the concept of a \emph{strong CE}.
\begin{definition}\label{def:strong_CE}
A strong counterfactual explanation for~\eqref{eq:knapsack_problem} is a point $(c,A,b)\in \H$ that satisfies the strong CE condition:
    \begin{equation}\label{eq:strong_CE_definition}
        \begin{aligned}
           \argmin_{x\in \X: A x\ge b} c^\top x \subseteq \D.
        \end{aligned}
    \end{equation}
\end{definition}
\noindent Once again, for a given distance measure $\delta$, the cost of a strong CE is then given by $\delta\left( (c,A,b), (\hat c, \hat A, \hat b)\right)$. We call a strong CE with minimal cost $\delta$ an \emph{optimal strong CE}. An optimal strong CE can be computed by solving the following pessimistic bilevel problem:
\begin{equation}\label{eq:strong_CE_problem}
    \begin{aligned}
            \min_{c,A,b,x} \ & \delta\left( (c,A,b), (\hat c, \hat A, \hat b)\right) \\
            s.t.\quad & x\in \D \quad \forall x\in \argmin_{x\in \X: A x\ge b} c^\top x, \\
            & (c,A,b) \in \H.
        \end{aligned}
\end{equation}
To illustrate the definitions of weak and strong CEs, let us give a toy example.
\begin{example}
Consider the following instance of \eqref{eq:knapsack_problem}:
\begin{align*}
    \min \ & x_1+2x_2+2x_3 \\
    s.t. \quad & x_1+3x_2+2x_3 \ge 3, \\
    & x\in \{0,1\}^3 .
\end{align*}
The optimal solution is to select only $x_2=1$ and to set all other variables to zero. Assume that only the constraint parameters $\hat a_2=3$ and $\hat a_3 = 2$ are mutable, each in the range $\{0, 1, 2, 3, 4\}$. Here is the counterfactual question: ``What is the minimal change in the mutable parameters we need to perform, such that a solution with $x_3 = 1$ is optimal, i.e.,
$\D = \{ x\in \{0,1\}^3: x_3 = 1 \}$?'' Figure \ref{fig:CE_example} shows the feasible regions for the weak and the strong CEs for this example.
\begin{figure}[H]
    \centering
    
\newcommand{\diagonalFill}[5]{%
  \begin{scope}
    \clip (#1,#2) rectangle ++(1,1);
    \fill[#3] (#1,#2+1) -- (#1+1,#2+1) -- (#1+1,#2) -- cycle; 
    \fill[#4] (#1,#2+1) -- (#1,#2) -- (#1+1,#2) -- cycle;     
    \draw[#5, thick] (#1,#2+1) -- (#1+1,#2);
  \end{scope}
}

\begin{tikzpicture}[scale=1]
\definecolor{outside}{RGB}{230,230,230}
\definecolor{strongCF}{RGB}{200,255,200}
\definecolor{weakCF}{RGB}{255,255,200}

\definecolor{cell00}{RGB}{230,230,230}
\definecolor{cell10}{RGB}{200,255,200}
\definecolor{cell20}{RGB}{200,255,200}
\definecolor{cell30}{RGB}{200,255,200}

\definecolor{cell01}{RGB}{200,255,200}
\definecolor{cell11}{RGB}{255,255,200}
\definecolor{cell21}{RGB}{200,255,200}
\definecolor{cell31}{RGB}{200,255,200}

\definecolor{cell02}{RGB}{230,230,230}
\definecolor{cell12}{RGB}{230,230,230}
\definecolor{cell22}{RGB}{255,230,200}
\definecolor{cell32}{RGB}{200,255,200}

\definecolor{cell03}{RGB}{230,230,230}
\definecolor{cell13}{RGB}{230,230,230}
\definecolor{cell23}{RGB}{230,230,230}
\definecolor{cell33}{RGB}{255, 230, 200}

\foreach \x in {0,1,2,3} {
    \foreach \y in {0,1,2,3} {
        \fill[cell\x\y] (\x,\y) rectangle ++(1,1);
    }
}

\diagonalFill{1}{1}{strongCF}{strongCF}{green}
\diagonalFill{0}{1}{strongCF}{outside}{green}
\diagonalFill{1}{0}{strongCF}{outside}{green}

\def\xstart{5.5}
\def\ystart{0}

\filldraw[fill=gray!30, draw=black] (\xstart,\ystart) rectangle ++(1,0.66);
\node[right] at (\xstart+1.1,\ystart+0.33) {Item $1$: $c_1=a_1=1$};

\filldraw[fill=gray!30, draw=black] (\xstart,\ystart+0.9) rectangle ++(1,1.9);
\node[right] at (\xstart+1.1,\ystart+1.9) {Item $2$: $c_2=2,\hat{a}_2=3, a_2\in [0,4]$};

\filldraw[fill=gray!30, draw=black] (\xstart,\ystart+3) rectangle ++(1,1.33);
\node[right] at (\xstart+1.1,\ystart+3.66) {Item $3$: $c_3=\hat{a}_3=2, a_3 \in [0,4]$};

\draw[thick,gray] (0,2) -- (2,2);
\draw[thick,orange] (2,2) -- (3,2);
\draw[thick,orange] (2,2) -- (2,3);
\draw[thick,green] (2,0) -- (2,2);
\draw[thick,green] (3,0) -- (3,3);
\draw[thick,gray] (0,3) -- (3,3);
\draw[thick,orange] (3,3) -- (4,3);
\draw[thick,orange] (3,3) -- (3,4);

\draw[->,thick] (-0.5,0) -- (4.5,0) node[right] {$a_3$};
\draw[->,thick] (0,-0.5) -- (0,4.5) node[above] {$a_2$};

\foreach \x in {1,...,4} {
    \node[below] at (\x,-0.1) {\x};
    \node[left] at (-0.1,\x) {\x};
}

\node[] at (0,0) {\textcolor{gray}{\textbf{\Large $\times$}}};
\node[] at (1,0) {\textcolor{gray}{\textbf{\Large $\times$}}};
\node[] at (2,0) {\textcolor{green}{\textbf{\Large $\times$}}};
\node[] at (3,0) {\textcolor{green}{\textbf{\Large $\times$}}};
\node[] at (4,0) {\textcolor{green}{\textbf{\Large $\times$}}};

\node[] at (0,1) {\textcolor{gray}{\textbf{\Large $\times$}}};
\node[] at (1,1) {\textcolor{green}{\textbf{\Large $\times$}}};
\node[] at (2,1) {\textcolor{green}{\textbf{\Large $\times$}}};
\node[] at (3,1) {\textcolor{green}{\textbf{\Large $\times$}}};
\node[] at (4,1) {\textcolor{green}{\textbf{\Large $\times$}}};

\node[] at (0,2) {\textcolor{gray}{\textbf{\Large $\times$}}};
\node[] at (1,2) {\textcolor{gray}{\textbf{\Large $\times$}}};
\node[] at (2,2) {\textcolor{orange}{\textbf{\Large $\times$}}};
\node[] at (3,2) {\textcolor{green}{\textbf{\Large $\times$}}};
\node[] at (4,2) {\textcolor{green}{\textbf{\Large $\times$}}};

\node[] at (0,3) {\textcolor{gray}{\textbf{\Large $\times$}}};
\node[] at (1,3) {\textcolor{gray}{\textbf{\Large $\times$}}};
\node[] at (2,3) {\textcolor{black}{\textbf{\Large O}}};
\node[] at (2,3) {\textcolor{gray}{\textbf{\Large $\times$}}};
\node[] at (3,3) {\textcolor{orange}{\textbf{\Large $\times$}}};
\node[] at (4,3) {\textcolor{orange}{\textbf{\Large $\times$}}};

\node[] at (0,4) {\textcolor{gray}{\textbf{\Large $\times$}}};
\node[] at (1,4) {\textcolor{gray}{\textbf{\Large $\times$}}};
\node[] at (2,4) {\textcolor{gray}{\textbf{\Large $\times$}}};
\node[] at (3,4) {\textcolor{orange}{\textbf{\Large $\times$}}};
\node[] at (4,4) {\textcolor{orange}{\textbf{\Large $\times$}}};

\node[rotate=-45] at (0.7,0.7) {};
\node[rotate=-45] at (1.2,1.2) {$i_1,i_2,i_3 (5)$};
\node[rotate=-45] at (1.7,1.7) {\footnotesize$i_2,i_3(4)$};
\node at (2.5,1) {\footnotesize$i_1,i_3(3)$};
\node at (1,2.5) {\footnotesize$i_1,i_2(3)$};
\node at (1.5,3.5) {$i_2(2)$};
\node at (2.5,2.5) {\footnotesize $i_1,i_2\lor i_3 (3)$};
\node[rotate=-45] at (2.3,0.3) {};
\node at (3.5,1.5) {$i_3 (2)$};
\node at (3.5,3.5) {\footnotesize$i_2 \lor i_3 (2)$};

\end{tikzpicture}
    \caption{Counterfactual regions for a \eqref{eq:knapsack_problem} with right-hand side $b=3$ and three item $i_1, i_2,i_3$, the latter two of which have a mutable weight. The favored solution space consists of all all solutions that contain $i_3$, i.e., $\mathcal D_1 = \{ x\in \{0,1\}^3: x_3 = 1 \}$. \textcolor{green}{Green} and \textcolor{orange}{orange} regions represent strong and weak CEs, respectively. Each region is labeled with the items that are part of an optimal solution and the optimum solution value. The black O marks the parameter values $(\hat{A})$ of the present problem.}
    \label{fig:CE_example}
\end{figure}
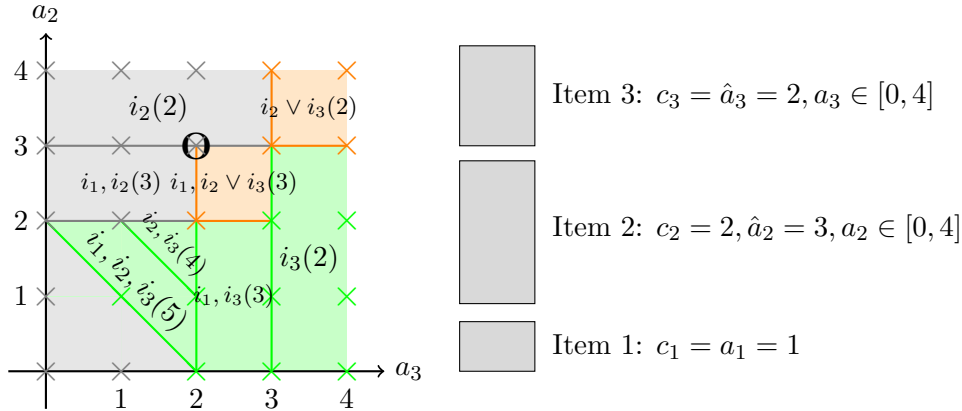
\end{example}

In the remainder of the paper, we present both results establishing the computational complexity of the problem of constructing weak and strong CEs, as well as algorithms for doing so. 
For reasons discussed later in the paper (see especially Remark~\ref{rem:more_mutable_constraints}), the algorithms presented assume $m = 1$ (we are only allowed to modify a single constraint). The complexity results, however, are general. 

\section{Complexity}

In this section, we show that the problems of deciding whether a counterfactual explanation exists is complete for $\Sigma_2^p$ in both the weak and strong cases. We use the oracle-based definition of the polynomial hierarchy, as formalized in~\citep{AroBar07}, and the polynomial many-to-one notion of reduction introduced by~\cite{karp1972reducibility}. This means that to prove membership in $\Sigma_2^p$, we have to show that whenever a weak/strong CE exists, there is a formal certificate of this fact that can be verified in polynomial time, given an oracle for solving problems in $\NP$. The assumption that we have access to an oracle for solving problems in $\NP$ means, in particular, that we assume that we have the ability to solve the decision version of an ILP in constant time, and hence, the optimization version of ILP in polynomial time. 

\subsection{Weak Counterfactual Explanations} \label{sec:weak_CE}

We first consider the problem of determining whether there exists a weak CE with respect to a present optimization problem of the form~\eqref{eq:knapsack_problem}, which we denote as $\wce$. Formal complexity arguments require careful specification of the input and its (encoded) size. For this problem, the formal input is the triple $(\hat{c}, \hat{A}, \hat{b})$ and the sets $\X$, $\H$, and $\D$. To reason about complexity and input sizes, we must make assumptions about the sets $\X$, $\H$, and $\D$, and we assume here that these sets are the feasible regions of ILPs (the results should hold for any sets for which the problem of determining whether the set is non-empty is in $\NP$). The size of the input is the sum of the sizes of the descriptions of the individual components. 

When a weak CE exists, Equation~\eqref{eq:weak_CE_definition} suggests both a certificate and a method of verifying it that requires solving an ILP. 
It remains to show that there must always exist such a certificate that is polynomial in the size of the input and that this certificate can be verified in polynomial time.    

\begin{theorem}\label{thm:insigma2p}
$\wce$ is in $\Sigma_2^p$.
\end{theorem}
\begin{proof}
If a weak CE exists, this means that there must exist a point $x^* \in \Z^n$ and a triple $(c^*,A^*,b^*) \in \Z^n \times \Z^{m \times n} \times \Z^n$ such that $x^* \in \argmin_{x\in \X: A^*x\ge b^*} (c^*)^\top x$ and that $x^* \in \D$. For $\tilde c:=(c^*)^\top x^*$ we must then have that $(x^*, c^*, A^*, b^*) \in \mathcal{Q}$ where
\[
\mathcal{Q} := \{(x,c,A,b) \in \X\cap \D \times \mathcal H : Ax\ge b, c^\top x = \tilde c\}
\]
Since mixed-integer quadratic programming feasibility is in $\NP$ (see Theorem 1 in~\cite{pia2017mixed}), there must exist $(\bar{x}, \bar{c}, \bar{A}, \bar{b}) \in \mathcal{Q}$ that has size polynomial in the original input data. We claim that $(\bar{x}, \bar{c}, \bar{A}, \bar{b}) \in \mathcal{Q}$ is then a certificate that can be checked in polynomial time, given an oracle for problems in $\NP$.

The verification requires checking first whether $\bar{x} \in \X\cap \D$, which can be done in polynomial time under our assumptions. Second, we need to check that $\bar{x} \in \argmin_{x\in \X: \bar{A}x\ge \bar{b}} \bar{c}^\top x$. This can also be done in polynomial time using the available $\NP$ oracle. 
\end{proof}

\noindent The following result then shows that $\wce$ is $\Sigma_2^p$-complete.

\begin{theorem}\label{thm:weak_CE_sigma-hard}
$\wce$ is complete for $\Sigma_2^p$
and cannot be approximated with a constant approximation factor in polynomial time, unless the polynomial hierarchy collapses, 
even when $\H = \{\hat c\} \times \H_a \times \{ \hat b\}$ (for given $\hat{c} \in \Z^n$ and $\hat{b} \in \Z^m$) and $\X=\{ 0,1\}^n$. 
\end{theorem}

\begin{proof}
We reduce the unitary bilevel interdiction knapsack problem (UBIK) to $\wce$. UBIK is defined as follows:  Let $(\bar c,\bar a,\bar b)$ be an instance of the binary knapsack problem with $n$ items
\begin{align*}
\max \ & \bar c^\top x \\
s.t. \quad & \bar a ^\top x \le \bar b,\\
& x\in\{ 0,1\}^n,
\end{align*}
where all entries of $(\bar c,\bar a,\bar b)$ are positive and integer.
Consider a follower, who desires a solution with objective value at least $K$, and a leader who can delete up to $k$ items from the knapsack. Then, the UBIK consists of deciding if there exists a subset of items $I \subseteq [n]$ with $|I| \le k$ such that for any feasible knapsack solution using only the items from $[n] \setminus I$, the objective value is strictly less than a given target $K$. It was shown in \cite{tomasaz2024} that UBIK is $\Sigma_2^p$-complete. The UBIK problem is equivalent to an instance of $\wce$ where the present problem is
\begin{align*}
\min \ & -\bar c^\top x \\
s.t. \quad & -\bar a ^\top x \ge -\bar b,\\
& x\in\{ 0,1\}^n,
\end{align*}

\noindent which has been put into the required form by negating the input data, and the required sets are defined as
\begin{itemize}
    \item $\X=\{ 0,1\}^n$,
    \item $\D = \{x \in \{ 0,1\}^n: -\bar{c}^\top x \ge -K + 1 \}$, and
    \item $\H = \{\bar c\} \times \H_a \times \{ \bar b\}$, where $\H_a = \{a \in \mathbb{R}^n: a_i = -\bar{a_i} - \delta_i\bar b, \ i\in [n], \ \sum_{i=1}^{n} \delta_i \le k, \ \delta\in \{ 0,1\}^n\}$.
\end{itemize}

The definition of $\D$ encodes the condition that the optimal values is strictly less than $K$, while the definition of $\H$ encodes that up to $k$ items may be forced to have value zero in the solution (interdicted). 

Now let $a \,\in\, \H_a$ be a weak CE. By construction, 
the maximum attainable objective of the knapsack problem with weights $a$ is $K-1$. Furthermore, at most $k$ items were interdicted (since $a\in \mathcal H_a$) to achieve that,  i.e., a YES instance for UBIK. Similarly, if the weak CE problem is infeasible this implies that for every set of at most size $k$ of interdicted items there exists a KP solution that has objective value at least $K$. Hence, we showed that UBIK is a yes-instance iff the constructed weak CE problem is feasible.

Since UBIK is $\Sigma_2^p$-complete, the reduction shows that $\wce$ is $\Sigma_2^p$-hard. 
Together with Theorem~\ref{thm:insigma2p}, this implies that $\wce$ is $\Sigma_2^p$-complete, even under strong restrictions. 
Notably, the CE only requires solving a feasibility problem, which implies that no efficient approximation algorithms exist unless the polynomial hierarchy collapses. 
\end{proof}

\begin{remark}
For the \eqref{eq:weak_CE_problem} where only the right hand side $b$ is mutable, it suffices to enumerate all relevant right hand side values of $b$. Assuming that $\H_b = \{ b\in \mathbb Z: \underline{b} \le b \le \bar b\}$, we then have to solve the problem
$$\begin{aligned}
    \min_{b \in \mathbb{Z},\,\underline{b} \le b \le \bar{b}} \ & \delta\left((\hat c, \hat A, b), (\hat c, \hat A, \hat b)\right) \\
    s.t.\quad & x \in \argmin_{x\in \X: \hat A x\ge b} \hat c^\top x, \\
    & x \in \D .
\end{aligned}$$
We can now iterate through all values of $b$, sorted by distance regarding $\delta\left((\hat c, \hat A, b), (\hat c, \hat A, \hat b)\right)$, and for each, we can check if the Condition \eqref{eq:weak_CE_definition} for $(\hat c, \hat A, b)$ is true. The latter condition can be verified by checking 
$$
\min_{x\in \X: \hat A x\ge b} \hat c^\top x = \min_{x\in \X\cap \D: \hat A x\ge b} \hat c^\top x,
$$ 
which involves solving two (potentially) $\NP$-hard problems.
\end{remark}

This remark leads to a solution algorithm when the present problem is a knapsack problem. If at the same time the favored solution space is given by a fixed number of knapsack constraints, then both problems we have to solve for every value of $b$ can be solved in pseudopolynomial time, which leads to a pseudopolynomial algorithm. The exact computational complexity discussion follows.

\begin{theorem}
    Let the underlying problem be a knapsack problem, i.e., $\X=\{ 0,1\}^n$, $\hat A=\hat a\in \mathbb Z_+^n$ and $$\D = \left\{x\in\{ 0,1\}^n:  q_i^\top x \ge p_i, \ i=1,\ldots ,T\right\}$$ with $q_i\in\mathbb Z_+^n$, $p_i\in\mathbb Z$ for all $i\in [T]$. Furthermore, assume $\H_b = \{ b\in \mathbb Z: \underline{b} \le b \le \bar b\}$, then \eqref{eq:weak_CE_problem} can be solved in time $\mathcal O\left( n \bar b^2 \prod_{i=1}^{T} (p_i+1)\right)$.
\end{theorem}
\begin{proof}
First, note that if $\underline b<0$, then we can remove all negative values for $b$. Hence, in the worst case we have to check $\bar b$ times whether $b$ is a weak CE. For each value of $b$, we have to solve two multi-dimensional knapsack problems:
\[
\min_{x\in\{ 0,1\}^n: \hat a^\top x \ge b} \hat c^\top x \quad \text{ and } \quad 
\min_{\substack{x\in\{ 0,1\}^n: \hat a^\top x\ge b, \\ q_i^\top x \ge p_i \ i=1,\ldots ,T}} \hat c^\top x .
\]
The second problem can be solved by the classical pseudopolynomial algorithm for multi-dimensional knapsack problems in time $\mathcal O\left( n b \prod_{i=1}^{T} (p_i+1) \right)$ (see \citep{freville2004multidimensional}) and this dominates the runtime of solving the classical knapsack problem on the left. Since $b\le \bar b$ and we have to solve the latter problems $\bar b$ times, the desired result follows.
\end{proof}

\subsection{Strong Counterfactual Explanations}

In this section, we show complexity results and solution algorithms for the problem of determining whether there exists a strong CE, which we refer to as $\sce$. The basic setup for $\sce$ is the same as for the case of $\wce$, and the problem has the same form of input. The difference is in the verification problem, which is more involved in the case of a strong CE. Rather than needing to show that $\exists x \in \D, x\in\argmin_{x\in \X: A x\ge b} c^\top x$, we need to show that $\forall x \in \D, x \in \argmin_{x\in \X: A x\ge b} c^\top x$.

We first provide the following proposition showing that for a given point $(c,a,b)\in \H$, we can check if the point is a strong CE by solving three $\NP$-hard optimization problems.
\begin{proposition}
A parameter vector $(c,A,b)\in\H$ is a strong counterfactual explanation, if and only if
\begin{equation}\label{eq:check_strong_CE}
\min_{x\in \X: A x\ge b} c^\top x = \min_{\substack{x\in \X\cap \D: A x\ge b}} c^\top x \leq \min_{\substack{x\in \X\setminus D: \\ A x\ge b}} c^\top x - 1.
\end{equation}
\end{proposition}
\begin{proof}
The first equation ensures that \emph{some} solution from $\D$ is indeed optimal for the parameters $(c,a,b)$, just as in the weak CE case. The second inequality ensures that \emph{all} optimal solutions are contained in $\D$. The result follows directly from the definition of strong CE.
\end{proof}
Note that the condition $x\notin \D$ used in the last of the three problems is not always possible to formulate compactly. 
If $\D$ is defined by a single constraint,  i.e., $\D_\ge = \{ x\in \X: \sum_{i\in\mathcal I} \alpha_i x_i \le \beta\}$, then the condition $x\notin \D$ is equivalent to $x\in \{ x\in \X: \sum_{i\in\mathcal I} \alpha_i x_i > \beta\}$. However, if the description of $\D$ is given by multiple constraints, a big-M formulation may need to be used to model $x\notin \D$. In the rest of the section, we assume that we have a polynomial size description for $x\notin \D$. Under this condition,  we show that $\sce$ is in $\Sigma_2^p$ by an argument similar to that for the weak CE, but with a different certificate. In the case that a strong CE exists, there are points $x_1^*, x_2^* \in \Z^n$ and a triple $(c^*,A^*,b^*) \in \Z^n \times \Z^{m \times n} \times \Z^n$ such that $x_1^* \in \argmin_{x\in \X: A^*x\ge b^*} (c^*)^\top x$, $x_1^* \in \D$, and $x_2^* \in \argmin_{x\in \X, x \not\in \D: A^*x\ge b^*} (c^*)^\top x$, with $(c^*)^\top x_1^* \leq (c^*)^\top x_2^* - 1$. As in $\wce$, we must show that a certificate of polynomial size, consisting of these elements, must exist and can be verified in polynomial time, given an oracle for problems in $\NP$. 
\begin{theorem}\label{thm:strong_insigma2p}
$\sce$ is in $\Sigma_2^p$.
\end{theorem}
\begin{proof}
If a strong CE exists, this means that there must exist $x_1^*, x_2^* \in \Z^n$ and a triple $(c^*,A^*,b^*) \in \Z^n \times \Z^{m \times n} \times \Z^n$ such that $x_1^* \in \argmin_{x\in \X: A^*x\ge b^*} (c^*)^\top x$, $x_1^* \in \D$, and $x_2^* \in \argmin_{x\in \X, x \not\in \D: A^*x\ge b^*} (c^*)^\top x$, with $(c^*)^\top x_1^* \leq (c^*)^\top x_2^* - 1$. For $\tilde c_1 := (c^*)^\top x_1^*$ we must then have
\begin{align*}
& (x_1^*, x_2^*, c^*, A^*, b^*) \in \mathcal{R} := \\
& \{(x_1,x_2,c,A,b) \in (\X \cap \D) \times \X \times \mathcal H:\\ & Ax_1\ge b, Ax_2\ge b, c^\top x_1 = \tilde c_1, \\&c^\top x_1 \leq c^\top x_2 - 1\}
\end{align*}
Since MIQP feasibility is in $\NP$, there must exist $(\bar{x}_1, \bar{x}_2, \bar{c}, \bar{A}, \bar{b}) \in \mathcal{R}$ that has size polynomial in the original input data (see Theorem 1 in~\citep{pia2017mixed}). We claim that $(\bar{x}_1, \bar{x}_2, \bar{c}, \bar{A}, \bar{b}) \in \mathcal{R}$ is then a certificate that can be checked in polynomial time, given an oracle for problems in $\NP$.

The verification requires checking first whether $\bar{x}_1 \in \X\cap\D$ and $\bar x_2\in\X$, which can be done in polynomial time under our assumptions. Second, we need to check both that $\bar{x}_1 \in \argmin_{x\in \X: \bar{A}x\ge \bar{b}} \bar{c}^\top x$ and whether $x_2^* \in \argmin_{x\in \X, x \not\in \D: \bar Ax\ge \bar b} \bar c^\top x$. Both of these checks can also be done in polynomial time using the available $\NP$ oracle. 
\end{proof}

\begin{theorem}\label{thm:strong_CE_sigma-hard}
$\sce$ is complete for $\Sigma_2^p$, even for $\H = \{\hat c\} \times \H_a \times \{ \hat b\}$, where $\H_a$ is a finite set, and $\X=\{ 0,1\}^n$. 
\end{theorem}
\begin{proof}
We use exactly the same reduction as in the proof of Theorem~\ref{thm:weak_CE_sigma-hard} with the same inputs, but we simply interpret those inputs as a description of an instance of $\sce$ instead of $\wce$. 
\end{proof}

\section{Algorithms}

Building on our earlier observations regarding the structure of CE problems, we now present algorithms designed to compute both weak and strong CEs.

\subsection{Weak Counterfactual Explanations}

In the following, we develop three algorithms corresponding to the three cases of the mutable parameter space listed in Section \ref{sec:problem_definition}.

\paragraph{Only objective function parameters are mutable.}
If $\hat a$ and $\hat b$ are fixed, the mutable parameter space is denoted by $\H_c \subseteq \mathbb  R^n$. In this case, we can reformulate \eqref{eq:weak_CE_problem} as
\begin{subequations}\label{eq:weak_CE_objective_mutable_reformulation}
\begin{align}
     \min_{x,c} \ & \delta\left( (c,\hat{a},\hat{b}), (\hat c, \hat a, \hat b)\right) \\
            s.t.\quad & c^\top x \le c^\top y \quad \forall y \in \X:  \hat a^\top y\ge \hat b, \label{eq:weak_CE_only_objective1}\\
            & \hat a^\top x\ge \hat b, \label{eq:weak_CE_only_objective2}\\
            & x \in \D\cap \X, \label{eq:weak_CE_only_objective3} \\
            & c\in \H_c,
\end{align}
\end{subequations}
where Constraints \eqref{eq:weak_CE_only_objective2} and \eqref{eq:weak_CE_only_objective3} ensure that an optimal solution $x$ is feasible and lies in $\D$, while Constraints \eqref{eq:weak_CE_only_objective1} ensure that the solution $x$ is optimal for the objective vector $c$. Note that the problem has one constraint for every feasible solution $y\in\X\cap \{ y: \hat a^\top y \ge \hat b\}$,  i.e., the formulation can be of exponential size. Hence, we iteratively generate the constraints by finding the most violating constraint via a separation problem in each iteration. This procedure is the same as in \cite{korikov2023objective} applied to our problem setup. The whole procedure can be found in Algorithm \ref{alg:solve_weak_CE_objective_mutable}. The correctness follows from the same argumentation as in \cite{korikov2023objective}.

\begin{algorithm}
\caption{Solve Problem \eqref{eq:weak_CE_objective_mutable_reformulation}}\label{alg:solve_weak_CE_objective_mutable}
\begin{algorithmic}
\Require $\hat a,\hat b,\hat c,\D, \X, \H$
\Ensure Optimal solution $c^*$ of \eqref{eq:weak_CE_objective_mutable_reformulation}.
\State set $\mathcal Y \leftarrow \emptyset$
\State optimal$\leftarrow$ false
\While{not optimal}
\State Calculate an optimal solution $(\tilde c, \tilde x)$ of problem
\begin{align*}
    \min_{c,x} \ & \delta\left( (c,\hat{a},\hat{b}), (\hat c, \hat a, \hat b)\right) \\
    s.t. \quad & c^\top x \le c^\top y \quad \forall y \in \mathcal Y, \\
    &\hat a^\top x\ge \hat b, \\
    & x \in \D\cap \X, \\
    & c\in \H_c.
\end{align*}
\State Calculate an optimal solution $\tilde y$ of the separation problem
\begin{align*}
\min \ & \tilde c^\top y \\
s.t. \quad & \hat a^\top y \ge \hat b, \\ 
& y\in\X.
\end{align*}
\If{$\tilde c^\top \tilde y < \tilde c^\top \tilde x$}
\State $\mathcal Y\leftarrow \mathcal Y\cup \{ \tilde y\}$
\Else
\State optimal $\leftarrow$ true \Comment{End While Loop}
\EndIf
\EndWhile
 \ \\
\Return $\tilde c$
\end{algorithmic}
\end{algorithm}

\paragraph{Only constraint parameters are mutable.}
In this subsection, only the constraint parameters $a,b$ are mutable,  i.e., $\hat c$ is fixed. The mutable parameter space is denoted as $\mathcal{H}_a \times \mathcal {H}_b$, or in short, $\H_{a,b} \subseteq \mathbb  Z^{n+1}$.

We first define the smallest and the largest value an optimal solution $x\in\D \cap \X$ can have over all parameter configurations $(a,b)\in \H_{a,b}$. The smallest value is
\begin{equation}\label{eq:c_min_hard} 
        c_{\min} := \min_{(a,b)\in\H_{\text{feas}}} \min_{\substack{a^\top x\ge b \\ x\in\D\cap \X}} \  \hat c^\top x
\end{equation}
and the largest value is
\begin{equation}\label{eq:c_max} 
c_{\max} := \max_{(a,b)\in\H_{\text{feas}}} \min_{\substack{a^\top x\ge b \\ x\in\D\cap \X}} \  \hat c^\top x,
\end{equation}
where $H_{\text{feas}}:=\left\{ (a,b)\in \H_{a,b}: \{ x\in \D\cap \X: a^\top x \ge b\} \neq \emptyset \right\}$. 
The following lemma provides a simple way to calculate bounds for $c_{\text{min}}$ and $c_{\text{max}}$.
\begin{lemma}\label{lem:bounds_cmin_cmax}
It holds
\[
c_{\min} \ge \min_{\substack{a_{\max}^\top x \ge b_{\min} \\ x\in \D \cap\X}} \hat c^\top x =: \underline{c}
\]
and 
\[
c_{\max} \le \max_{\substack{a_{\max}^\top x \ge b_{\min} \\ x\in \D \cap\X}} \hat c^\top x =:\bar c,
\]
where $(a_{\max})_i := \max_{(a,b)\in \H_{a,b}} a_i$ and $b_{\min} := \min_{(a,b)\in \H_{a,b}} b$.
\end{lemma}
\begin{proof}
To prove the first result consider an arbitrary parameter configuration $(a,b)\in\H_{\text{feas}}$ and any feasible solution $x^*$,  i.e., it holds $a^\top x^* \ge b$. Since $(\hat a, \hat b)\in \H_{a,b}$ such a point $(a,b)$ always exists. Then this solution is feasible for $a_{\max}^\top x \ge b_{\min}$ as well, since
\[
a_{\max}^\top x^* \ge a^\top x^* \ge b \ge  b_{\min},
\]
where the first inequality follows from $x^*\ge 0$ and the definition of $a_{\max}$. Hence, it holds for every $(a,b)\in\H_{\text{feas}}$ that
\[
\min_{\substack{a^\top x\ge b \\ x\in\D\cap \X}} \  \hat c^\top x \ge \min_{\substack{a_{\max}^\top x \ge b_{\min} \\ x\in \D \cap\X}} \hat c^\top x, 
\]
which proves the first result. The second result follows from a similar argumentation. Applying the results above for all $(a,b)\in \H_{\text{feas}}$, we obtain
\[
\min_{\substack{a^\top x\ge b \\ x\in\D\cap \X}} \  \hat c^\top x \le  \max_{\substack{a^\top x\ge b \\ x\in\D\cap \X}} \  \hat c^\top x  \le \max_{\substack{a_{\max}^\top x \ge b_{\min} \\ x\in \D \cap\X}} \hat c^\top x ,
\]
which proves the result.
\end{proof}

Since $\hat c$ and all solutions $x\in\X$ have only integer entries, the objective value of the best solution in $\D\cap \X$ for any $(a,b)\in \H_{\text{feas}}$ is contained in
\begin{equation}
\label{eqn:setC}
\mathcal C := \{ c_{\min}, c_{\min} + 1, \ldots , c_{\max}\}.
\end{equation}

The main idea of the algorithm developed in this section is to iterate over the values in $\mathcal C$ and find the best parameter setting $(a,b)\in\H_{a,b}$ that ensures the existence of a optimal solution in $\D$  with the corresponding optimal value from $\mathcal C$. We first show the following lemma.

\begin{lemma}\label{lem:feasible_weak_CE}
For every $v\in \mathcal C$, the following problem is either infeasible or any optimal solution $(a^*(v),b^*(v))$ is a (not necessarily optimal) weak counterfactual explanation:
\begin{subequations}
\begin{align}
    \min_{a,b,x} \ & \delta\left( (\hat{c},a,b), (\hat c, \hat a, \hat b)\right) \\
    s.t. \quad & a^\top y \le  b - 1  \quad \forall y\in \X: \hat c^\top y \le v -1, \label{eq:master_problem_1}\\
    &\hat c^\top x = v, \label{eq:master_problem_2}\\
    & a^\top x \ge b, \label{eq:master_problem_3}\\
    & x\in\D\cap \X, \label{eq:master_problem_4}\\
    & (a,b)\in \H_{a,b}. \label{eq:master_problem_5}
\end{align}
\label{eq:master_problem}
\end{subequations}
\end{lemma}
\begin{proof}
Assume the problem is feasible. Since $\H_{a,b}$ is discrete and bounded the problem must have an optimal solution $(a^*(v),b^*(v),x^*)$. We have to show that this optimal solution is a counterfactual explanation,  i.e., there exists a solution $\tilde x\in \D$ with
\[
\tilde x \in \argmin_{x\in \X: a^*(v)^\top x\ge b^*(v)} \hat{c}^\top x.
\]
We show in the following that $x^*$ fulfills the latter condition. First, Constraint \eqref{eq:master_problem_4} ensures that $x^*\in \D\cap \X$. Constraint \eqref{eq:master_problem_3} ensures that \[
a^*(v)^\top x^* \ge b^*(v)
\]
holds. Second, Constraints \eqref{eq:master_problem_1} and \eqref{eq:master_problem_2} ensure that every solution in $\X$ which has a better objective value than $\hat c^\top x^*$ is infeasible. Hence, $x^*$ must be optimal for the problem, and $(a^*(v),b^*(v))$ is a weak counterfactual explanation.
\end{proof}

The intuition behind Lemma \ref{lem:feasible_weak_CE} is as follows: The first set of constraints ensures that all solutions $y\in \X$ which are better than the considered optimal value $v$ are infeasible for the solution $(a^*(v),b^*(v))$,  i.e., the optimal value of the problem for the parameters $(\hat c,a^*(v),b^*(v))$ is at least $v$. The remaining constraints ensure that a solution $x\in\D$ exists which is feasible for the constraint parameters $(a^*(v),b^*(v))$ and has optimal value $v$. If such a solution exists, this solution must be optimal, and hence $(a^*(v),b^*(v))$ is a weak CE and a possible candidate for an optimal weak CE.

To solve Problem \eqref{eq:master_problem} effectively, a cut generation procedure can be performed, where the constraints for a solution $y$ are separated iteratively. This procedure is shown in Algorithm \ref{alg:solve_master_problem}.
Note that if in Algorithm \ref{alg:solve_master_problem} an upper bound $d^*$ for the cost of a CE is known, we use that bound to cut off non-improving solutions.
From classical constraint generation theory the following theorem directly follows.

\begin{theorem}\label{thm:correctness_algo_MP}
Algorithm \ref{alg:solve_master_problem} calculates an optimal solution of Problem \eqref{eq:master_problem}.
\end{theorem}

\begin{algorithm}
\caption{Solve Problem \eqref{eq:master_problem}}\label{alg:solve_master_problem}
\begin{algorithmic}
\Require $v, \hat a,\hat b,\hat c,\D, \X, \H_{a,b}$, known upper bound $d^*$ for \eqref{eq:weak_CE_problem}
\Ensure Optimal solution $(a^*(v),b^*(v))$ of \eqref{eq:master_problem}.
\State set $\mathcal Y \leftarrow \emptyset$
\State optimal$\leftarrow$ false
\While{not optimal}
\State Calculate an optimal solution $(\tilde a, \tilde b)$ of problem
\begin{align*}
    \min_{a,b,x} \ & \delta\left( (\hat{c},a,b), (\hat c, \hat a, \hat b)\right) \\
    s.t. \quad & a^\top y \le b - 1 \quad \forall y\in \mathcal Y, \\
    &\hat c^\top x = v, \\
    & a^\top x \ge b, \\
    & x\in\D\cap \X, \\
    & \delta\left( (\hat{c},a,b), (\hat c, \hat a, \hat b)\right) \le d^*, \\
    & (a,b)\in \H_{a,b} .
\end{align*}
\State If problem is infeasible, stop and return \textit{infeasible}.
\State Otherwise, calculate an optimal solution $\tilde y$ of the separation problem
\begin{align*}
\max \ & \tilde a^\top y \\
s.t. \quad & \hat c^\top y \le v - 1, \\ 
& y\in\X.
\end{align*}
\If{$\tilde a^\top \tilde y \ge \tilde b$}
\State $\mathcal Y\leftarrow \mathcal Y\cup \{ \tilde y\}$
\Else
\State optimal $\leftarrow$ true \Comment{End While Loop}
\EndIf
\EndWhile
 \ \\
\Return $(\tilde a, \tilde b)$
\end{algorithmic}
\end{algorithm}

We now show that there must exist a value $v^*\in\mathcal C$ such that $(a^*(v^*),b^*(v^*))$ is an optimal weak counterfactual explanation.

\begin{lemma}\label{lem:weak_CE_optimal_v}
There exists a $v^*\in \mathcal C$, such that any optimal solution $(a^*(v^*),b^*(v^*))$ of Problem \eqref{eq:master_problem} with $v=v^*$ is an optimal weak CE. 
\end{lemma}
\begin{proof}
Let $(a^*,b^*)\in\H_{a,b}$ be an optimal weak CE,  i.e., an optimal solution of Problem \eqref{eq:weak_CE_problem}. Then there exists a solution $x^*\in \D$ which is optimal for the problem 
\begin{align*}
    v^*:= \min & \ \hat c ^\top x \\
    s.t. \quad & (a^*)^\top x \ge b^*, \\
    & x\in\X .
\end{align*}
Since $(a^*,b^*)\in \H_{\text{feas}}$, $c_{\min} \le v^*\le c_{\max}$ must hold. And since $\hat c\in \mathbb Z^n$ and $\X\subset \mathbb Z_+^n$, $v^*$ is integer and it must hold that $v^*\in \mathcal C$. Hence, $(a^*,b^*,x^*)$ is feasible for Constraints \eqref{eq:master_problem_2} -- \eqref{eq:master_problem_5} for $v=v^*$. Furthermore, from the optimality of $x^*$ and the integrality of $y$ and $\hat c,a^*,b^*$, it follows that
\[
(a^*)^\top y \le b^* -1 \quad \forall y\in\X: \hat c^\top y \le  v^* -1.
\]
We can conclude that $(a^*,b^*,x^*)$ is feasible for Problem \eqref{eq:master_problem} with $v=v^*$. Since every optimal solution of \eqref{eq:master_problem} is a weak CE by Lemma \ref{lem:feasible_weak_CE}, it follows that $(a^*(v^*),b^*(v^*))$ must be an optimal weak CE.
\end{proof}

The latter two lemmas indicate that we can solve Problem \eqref{eq:weak_CE_problem} by iterating over all values $v\in\mathcal C$ and solve the corresponding Problem \eqref{eq:master_problem}. The best solution $(a^*(v),b^*(v))$ found over all these values must then be an optimal solution of \eqref{eq:weak_CE_problem}. While this idea is feasible, it may happen that $\mathcal C$ contains a large number of values, and solving \eqref{eq:weak_CE_problem} for each is too expensive. To cut off unpromising solutions, we derive a lower bound for the objective function values in the following lemma.

\begin{lemma}\label{lem:lower_bound}
For a given $\bar v\in \mathbb Z$ the optimal value of the following problem is a lower bound for every feasible master problem \eqref{eq:master_problem} with integer $v\ge \bar v$:
\begin{equation}\label{eq:lower_bound_problem}
\begin{aligned}
    \min_{a,b} \ & \delta\left( (\hat{c},a,b), (\hat c, \hat a, \hat b)\right) \\
    s.t. \quad & a^\top y \le b - 1 \quad \forall y\in \X: \hat c^\top y \le \bar v -1, \\
    & (a,b)\in \H_{a,b}.
\end{aligned}
\end{equation}
\end{lemma}
\begin{proof}
Consider an arbitrary $v\ge \bar v$ for which \eqref{eq:master_problem} is feasible and a corresponding arbitrary feasible solution $(a(v),b(v),x(v))$ of \eqref{eq:master_problem}. We show that $(a(v),b(v))$ is feasible for \eqref{eq:lower_bound_problem} to prove the result. Clearly, $(a(v),b(v))\in\H_{a,b}$ is true. Furthermore, from Constraints \eqref{eq:master_problem_1} it follows that 
\[
a(v)^\top y \le b(v) - 1 \quad \forall y\in \X: \hat c^\top y \le v-1.
\]
Since $\bar v\le v$, $(a(v),b(v))$ must also be feasible for \eqref{eq:lower_bound_problem}. This proves the result.
\end{proof}

Problem \eqref{eq:lower_bound_problem} can be solved again by a constraint generation procedure, where we iteratively generate constraints for corresponding solutions $y\in \X$.
An important property of Problem \eqref{eq:lower_bound_problem} is the following:
\begin{corollary}
    The lower bound provided by Lemma \ref{lem:lower_bound} is non-decreasing in $\bar v$.
\end{corollary}
\begin{proof}
This follows immediately from the fact that for $v \ge \bar v$:
$$\{y\in \X: \hat c^\top y \le \bar v -1\} \subseteq \{y\in \X: \hat c^\top y \le  v -1\}$$
\end{proof}

The resulting lower bounds can be used as a stopping criterion when iterating over all values $v\in\mathcal C$. More precisely, iterate through $\mathcal C$ starting with $v=c_{\min}$ and increasing it. For each $v\in\mathcal C$, we can solve \eqref{eq:master_problem} to obtain a feasible weak CE (if the corresponding problem is feasible). We can store the best feasible weak CE $(a_{\text{best}},b_{\text{best}})$ and the corresponding distance $d_{\text{best}}=\delta(a_{\text{best}},\hat a) + \delta(b_{\text{best}},\hat b)$. At the same time, for the current $v$, we can calculate the lower bound \eqref{eq:lower_bound_problem} with $\bar v = v$. If this lower bound is at least as large as the best known value $d_{\text{best}}$, we can terminate early and $(a_{\text{best}},b_{\text{best}})$ is an optimal weak CE. The whole procedure is presented in Algorithm \ref{alg:weak_only_constraint_mutable}.

\begin{algorithm}
\caption{Optimal Weak CE for Mutable Constraint Parameters.}\label{alg:weak_only_constraint_mutable}
\begin{algorithmic}
\Require $\hat a,\hat b,\hat c,\D, \H_{a,b}$
\Ensure Optimal weak CE $(a^*,b^*)$.
\State Best known solution: $(a^*,b^*) = \emptyset$
\State Best known distance: $d^* = \infty$
\State Best known lower bound: $lb = -\infty$
\State Calculate bounds $\underline{c}$, $\bar{c}$ as in Lemma \ref{lem:bounds_cmin_cmax}.
\ForAll{$v= \underline{c}, \underline{c} + 1,  \ldots , \bar{c}$}
\State Set $lb$ to the optimal value of \eqref{eq:lower_bound_problem} with $\bar v = v$.
\If{$lb \ge d^*$}
\State Stop and \textbf{return} $(a^*,b^*)$.
\Else
\State Calculate opt. solution $(a^*(v),b^*(v))$ of Problem \eqref{eq:master_problem} by Algorithm \ref{alg:solve_master_problem}. 
\If{
$\delta(a^*(v),\hat a) + \delta(b^*(v),\hat b) < d^*$}
\State $(a^*,b^*) \leftarrow (a^*(v),b^*(v))$
\State $d^*\leftarrow \delta(a^*(v),\hat a) + \delta(b^*(v),\hat b)$
\EndIf
\EndIf
\EndFor \\
\Return $(a^*,b^*)$
\end{algorithmic}
\end{algorithm}

\begin{remark}
  Before running Algorithm \ref{alg:weak_only_constraint_mutable}, we 
  verify whether $(\hat a, \hat b)$ is already a weak CE, which then must be the optimal one. This verification  involves solving two deterministic integer optimization problems as described in Section \ref{sec:weak_CE}. If the check is not successful, we run Algorithm \ref{alg:weak_only_constraint_mutable}.
\end{remark}

\begin{theorem}\label{theorem:weak_CE_ab_correctness}
    Algorithm \ref{alg:weak_only_constraint_mutable} calculates an optimal weak CE for the case where only the parameters $a$ and $b$ are mutable.
\end{theorem}
\begin{proof}
    The correctness follows from Lemma \ref{lem:bounds_cmin_cmax}, \ref{lem:feasible_weak_CE} and \ref{lem:lower_bound} and Theorem \ref{thm:correctness_algo_MP}.
\end{proof}

\paragraph{All parameters are mutable.}
In this section, all parameters in the objective and the constraint are mutable. In this case the weak CE problem \eqref{eq:weak_CE_problem} can be formulated as follows:
\begin{align*}
    \min_{c,a,b,x} \ & \delta\left( (c,a,b), (\hat c, \hat a, \hat b)\right) \\
    s.t. \quad & c^\top x \le c^\top y \quad \forall y\in \X: a^\top y \ge b,\\
    & a^\top x \ge b, \\
    & x\in \D \cap \X, \\
    & (c,a,b)\in \H.
\end{align*}
Since the number and type of constraints depend on the decision variables $a$ and $b$, the algorithms from the previous sections are not applicable. However, \eqref{eq:weak_CE_problem} can be reformulated as
\begin{subequations}\label{eq:weak_CE_all_mutable_big-M_reformulation}
    \begin{align}
    \notag\min_{c,a,b,x,z} \ & \delta\left( (c,a,b), (\hat c, \hat a, \hat b)\right) \\
    s.t. \quad & c^\top x \le c^\top y + M z_y \quad \forall y\in \X, \label{eq:constraints_all_parameters_big-M_1}\\
    & a^\top y \ge b - Mz_y \quad \forall y\in \X, \label{eq:constraints_all_parameters_big-M_2}\\
    &a^\top y \le b - 1 + M(1-z_y) \quad \forall y\in \X, \label{eq:constraints_all_parameters_big-M_3}\\
    & x\in \D\cap \X, \label{eq:constraints_all_parameters_big-M_4}\\
    & a^\top x \ge b, \label{eq:constraints_all_parameters_big-M_5}\\
    & z_y\in \{ 0,1\} \quad \forall y\in \X,\\
    & (c,a,b)\in \H,
\end{align}
\end{subequations}
where $M$ is a sufficiently large value. Constraints \eqref{eq:constraints_all_parameters_big-M_2} and \eqref{eq:constraints_all_parameters_big-M_3} ensure that $z_y=0$ if and only if $y$ is feasible for the choice of constraint parameters $(a,b)$, and $z_y=1$, otherwise. Then, Constraints \eqref{eq:constraints_all_parameters_big-M_1} ensure that the chosen solution $x$ has the best objective value below all feasible solutions $y \in \X$, while for all infeasible $y \in \X$ the constraint is redundant due to the big-M value. 
Constraints \eqref{eq:constraints_all_parameters_big-M_4} and \eqref{eq:constraints_all_parameters_big-M_5} ensure that the solution $x$ is feasible for the chosen parameters $a,b$ and lies in $\D$. 

For linear representable distance functions $\delta$, formulation \eqref{eq:weak_CE_all_mutable_big-M_reformulation} is bilinear, an thus can be solved by state-of-the-art optimization solvers. However, it can contain an exponential number of constraints. Similar to the previous algorithms it can be solved by iteratively generating solutions $y\in\X$, corresponding variables $z_y$ and corresponding constraints. The procedure is presented in Algorithm \ref{alg:solve_weak_CE_all_parameters}.

\begin{algorithm}
\caption{Solve Problem \eqref{eq:weak_CE_all_mutable_big-M_reformulation}}\label{alg:solve_weak_CE_all_parameters}
\begin{algorithmic}
\Require $\hat a,\hat b,\hat c,\D, \X, \H$
\Ensure Optimal solution $(a^*,b^*,c^*)$ of \eqref{eq:weak_CE_all_mutable_big-M_reformulation}.
\State set $\mathcal Y \leftarrow \emptyset$
\State optimal$\leftarrow$ false
\While{not optimal}
\State Calculate an optimal solution $(\tilde a, \tilde b, \tilde c, \tilde x)$ of problem
    \begin{align*}
    \min_{c,a,b,x,z} \ & \delta\left( (c,a,b), (\hat c, \hat a, \hat b)\right) \\
    s.t. \quad & c^\top x \le c^\top y + M z_y \quad \forall y\in \mathcal Y, \\
    & a^\top y \ge b - Mz_y \quad \forall y\in \mathcal Y, \\
    &a^\top y \le b - 1 + M(1-z_y) \quad \forall y\in \mathcal Y, \\
    & x\in \D\cap \X, \\
    & a^\top x \ge b, \\
    & z_y\in \{ 0,1\} \quad \forall y\in \mathcal Y,\\
    & (c,a,b)\in \H.
\end{align*}
\State Calculate an optimal solution $\tilde y$ of the separation problem
\begin{align*}
\min \ &\tilde c^\top y \\
s.t. \quad &\tilde a^\top y \ge \tilde b, \\
& y\in\X.
\end{align*}
\If{$\tilde c^\top \tilde y < \tilde c^\top \tilde x$}
\State $\mathcal Y\leftarrow \mathcal Y\cup \{ \tilde y\}$
\Else
\State optimal $\leftarrow$ true \Comment{End While Loop}
\EndIf
\EndWhile
 \ \\
\Return $(\tilde a, \tilde b, \tilde c)$
\end{algorithmic}
\end{algorithm}

\begin{theorem}
    Algorithm \ref{alg:solve_weak_CE_all_parameters} calculates an optimal weak CE.
\end{theorem}
\begin{proof}
    This follows analogously to Theorem \ref{theorem:weak_CE_ab_correctness}.
\end{proof}

\begin{remark}\label{rem:more_mutable_constraints}
We restricted the analysis of this section to the case of a single mutable constraint. The reason is that in the master problems \eqref{eq:master_problem} and \eqref{eq:weak_CE_all_mutable_big-M_reformulation} we model the condition that certain solutions $y$ are infeasible. For multiple mutable constraints this would involve modeling the condition $Ax \ngeq b$ which involves the introduction of big-M constraints, leading to the equivalent formulation
\begin{align*}
    &a_i^\top x \le b - 1 + M(1-z_i), \quad i=1,\ldots ,m, \\
    &\sum_{i=1}^{m} z_i \ge 1, \\
    &z\in \{ 0,1\}^n.
\end{align*}
In our iterative framework, this results in generating a large set of big-M constraints and corresponding binary variables. Since the resulting methods are computationally intractable even for few dimensions, we refrain from presenting this case in detail. On the other hand, for the case of a single mutable constraint the condition $a^\top x \ngeq b$ can be easily formulated as $a^\top x \le b - 1$.
\end{remark}

\subsection{Strong Counterfactual Explanations} \label{sec:strong_CE}

As in the case of weak CE algorithms, we can develop three algorithms corresponding to the three cases of the mutable parameter space listed in Section \ref{sec:problem_definition}.

\paragraph{Only objective function parameters are mutable.}
In this subsection, only the objective parameters are mutable,  i.e., $\hat a$ and $\hat b$ are fixed. The mutable parameter space is denoted as $\H_c \subset \mathbb  R^n$. In this case, we can reformulate \eqref{eq:strong_CE_problem} as
\begin{subequations}\label{eq:strong_CE_objective_mutable_reformulation}
\begin{align}
     \min_{x,c} \ & \delta\left( (c,\hat{a},\hat{b}), (\hat c, \hat a, \hat b)\right) \\
            s.t.\quad &  c^\top x \le c^\top y - 1 \quad \forall y \in \X\setminus \D:  \hat a^\top y\ge \hat b, \label{eq:strong_CE_only_objective1} \\
            & \hat a^\top x\ge \hat b, \label{eq:strong_CE_only_objective2}\\
            & x \in \D\cap \X, \label{eq:strong_CE_only_objective3} \\
            & c\in \H_c,
\end{align}
\end{subequations}
where Constraints \eqref{eq:strong_CE_only_objective2} and \eqref{eq:strong_CE_only_objective3} ensure that a solution $x$ from $\D$ is feasible, while Constraints \eqref{eq:strong_CE_only_objective1} ensure that any feasible solution in $\X\setminus \D$ cannot be optimal. It follows that for an optimal solution $c^*$ of \eqref{eq:strong_CE_objective_mutable_reformulation} all optimal solutions must be contained in $\D$. Note that \eqref{eq:strong_CE_objective_mutable_reformulation} can have one constraint for every feasible solution $y\in\X\setminus \D$, of which there can be exponentially many. Hence, similar to Algorithm \ref{alg:solve_weak_CE_objective_mutable} we iteratively generate the constraints by finding the most violating constraint via the separation problem
\begin{align*}
    \min \ & c^\top y \\
    s.t. \quad & \hat a^\top y \ge \hat b, \\
    & y\in \X\setminus \D.
\end{align*}

Note that the essential difference to the weak CE algorithm is the constraint $y\in \X\setminus \D$ in the separation problem. For simple cases of $\D$ (e.g., if $\D$ is described by a single constraint or by variable fixations), this condition can be modeled by linear constraints. However, in general this condition may involve the use of big-M constraints, similarly as described in Remark \ref{rem:more_mutable_constraints}.

\paragraph{Only constraint parameters are mutable.}
The ideas for the weak CEs can be extended to the strong CE case. We assume again that $\hat c$ only has integer entries and define the set $\mathcal C$ as in \eqref{eqn:setC}. Analogously to Lemma \ref{lem:feasible_weak_CE}, we can verify the following result: For every $v\in \mathcal C$, the following problem is either infeasible or an optimal solution thereof is a (not necessarily optimal) strong counterfactual explanation:
\begin{subequations}
\begin{align}
    \min_{a,b,x} \ & \delta\left( (\hat{c},a,b), (\hat c, \hat a, \hat b)\right) \\
    s.t. \quad &  a^\top y \le  b - 1 \quad \forall y\in \X\setminus \D: \hat c^\top y \le v, \label{eq:master_problem_strong1}\\
    &\hat c^\top x = v, \label{eq:master_problem_strong2}\\
    & a^\top x \ge b, \label{eq:master_problem_strong3}\\
    & x\in\D\cap \X, \label{eq:master_problem_strong4}\\
    & (a,b)\in \H_{a,b}. \label{eq:master_problem_strong5}
\end{align}
\label{eq:master_problem_strongCE}
\end{subequations}
If the strong CE problem is feasible, Constraints \eqref{eq:master_problem_strong2} -- \eqref{eq:master_problem_strong4} ensure that there exists a solution $x\in \D$ which is feasible and this solution has objective value $v$. Constraints \eqref{eq:master_problem_strong1} ensure that every solution, which is not in $\D$ and has objective value at least as good as $v$, is not feasible. Hence, feasible solutions outside of $\D$ are not optimal. 

The same result of Lemma \ref{lem:weak_CE_optimal_v} holds for the strong CE version, i.e., there exists a $v^*\in\mathcal C$, such that any optimal solution of \eqref{eq:master_problem_strongCE} is an optimal solution of \eqref{eq:strong_CE_problem}. Likewise, similar to the lower bound in Lemma \ref{lem:lower_bound}, it is easy to see that the following problem provides a lower bound for every problem \eqref{eq:master_problem_strongCE} with integer $v\ge \bar v$:

\begin{subequations}
\begin{align}
    \min_{a,b} \ & \delta\left( (\hat{c},a,b), (\hat c, \hat a, \hat b)\right) \\
    s.t. \quad &  a^\top y \le  b - 1  \quad \forall y\in \X\setminus \D: \hat c^\top y \le \bar v, \label{eq:lower_bound_strong1}\\
    & (a,b)\in \H_{a,b}.\label{eq:lower_bound_strong5}
\end{align}
\label{eq:lower_bound_strongCE}
\end{subequations}

Hence Algorithm \ref{alg:solve_master_problem} and Algorithm \ref{alg:weak_only_constraint_mutable} can be directly adapted for the strong CE case by replacing each subproblem with the adapted version presented above. The details of the corresponding algorithms are presented in the Appendix as Algorithm \ref{alg:solve_master_problem_strong} and Algorithm \ref{alg:strong_only_constraint_mutable}.

\paragraph{All parameters are mutable.}

In this section all parameters in the objective and the constraints are mutable. In this case, the strong CE problem \eqref{eq:strong_CE_problem} can be reformulated as
\begin{subequations}\label{eq:strong_CE_all_params_formulation}
\begin{align}
    \min_{c,a,b,x} \ & \delta\left( (c,a,b), (\hat c, \hat a, \hat b)\right) \\
    s.t. \quad & c^\top x \le c^\top y -1 \quad \forall y\in \X\setminus \D: a^\top y \ge b, \label{eq:strong_all_parameters_1}\\
    & a^\top x \ge b, \\
    & x\in \D \cap \X, \\
    & (c,a,b)\in \H.
\end{align}
\end{subequations}
The main difference to the weak CE case is that \eqref{eq:strong_all_parameters_1} ensures that all solutions which are not in $\D$ cannot be optimal. 
As for the weak CEs, we can then reformulate Problem \eqref{eq:strong_CE_all_params_formulation} as

\begin{subequations}\label{eq:strong_CE_all_mutable_big-M_reformulation}
    \begin{align}
    \min_{c,a,b,x,z} \ & \delta\left( (c,a,b), (\hat c, \hat a, \hat b)\right) \\
    s.t. \quad & c^\top x \le c^\top y - 1 + M z_y \quad \forall y\in \X\setminus \D, \label{eq:strong_constraints_all_parameters_big-M_1}\\
    & a^\top y \ge b - Mz_y \quad \forall y\in \X\setminus \D, \label{eq:strong_constraints_all_parameters_big-M_2}\\
    &a^\top y \le b - 1 + M(1-z_y) \quad \forall y\in \X \setminus \D, \label{eq:strong_constraints_all_parameters_big-M_3}\\
    & x\in \D\cap \X, \label{eq:strong_constraints_all_parameters_big-M_4}\\
    & a^\top x \ge b, \label{eq:strong_constraints_all_parameters_big-M_5}\\
    & z_y\in \{ 0,1\} \quad \forall y\in \X,\\
    & (c,a,b)\in \H,
\end{align}
\end{subequations}
where $M$ is a large enough big-M value. Constraints \eqref{eq:strong_constraints_all_parameters_big-M_2} and \eqref{eq:strong_constraints_all_parameters_big-M_3} ensure that $z_y=0$ if and only if $y$ is feasible for the choice of constraint parameters $(a,b)$ and $z_y=1$ otherwise. Hence, Constraints \eqref{eq:strong_constraints_all_parameters_big-M_1} ensure that the chosen solution $x$ has the best objective value below all feasible solutions $y \in \X$ while for all infeasible $y \in \X$ the constraint is redundant due to the big-M value. 
Constraints \eqref{eq:strong_constraints_all_parameters_big-M_4} and \eqref{eq:strong_constraints_all_parameters_big-M_5} ensure that the selected solution $x$ is feasible for the chosen parameters $a,b$ and lies in $\D$. 

Analogously to the weak CEs, formulation \eqref{eq:strong_CE_all_mutable_big-M_reformulation} is bilinear and  can be solved by state-of-the-art optimization solvers. However, it can contain an exponential number of constraints. Similar to the previous algorithms it can be solved by iteratively generating solutions $y\in\X$, corresponding variables $z_y$ and corresponding constraints. The procedure is presented in Algorithm \ref{alg:solve_strong_CE_all_parameters} in the Appendix.

\section{Computational Study}\label{sec:computational_results}

In this section, we perform numerical experiments on a set of test problems to evaluate the performance of the proposed algorithms in obtaining both weak and strong CEs. 

\subsection{Setup of Knapsack Instances}

\paragraph{Data.} We tested the algorithms on knapsack benchmark instances from \href{https://github.com/likr/kplib}{\texttt{kplib}}\footnote{See \href{https://github.com/likr/kplib}{https://github.com/likr/kplib}.}, which we interpreted as cover instances to be consistent with the notation throughout this work.
The instances were originally designed according to the procedure outlined in Kellerer, Pferschy \& Pisinger~\citep{kellerer2004exact}.
We used instances with a data range of $1000$\footnote{Note that Algorithms \ref{alg:weak_only_constraint_mutable}, \ref{alg:solve_weak_CE_all_parameters}, \ref{alg:strong_only_constraint_mutable} and \ref{alg:solve_strong_CE_all_parameters} are sensitive to changes in the data range by design,  i.e., changes in the data range will lead to a proportional change in the number of subproblems, and thus the runtime. } and instance sizes in $[10,20,30,40]$, which were generated by slicing larger instances.
Since it is well-known that the practical hardness of knapsack problems depends on the instance structure ~\citep{Pisinger2005,SmithMiles2021}, e.g., if items value/weight ratios are correlated, the problems tend to be harder to solve in practice~\citep{Pisinger2005}, we use both $20$ uncorrelated and strongly correlated instances each.
In total, we consider $4\times2\times20=160$ testing instances per favored solution space and CE type.

\paragraph{Favored solution spaces.}
Let $\mathcal I^+,\mathcal I^- \subset \{ 1,\ldots ,n\}$ with $\mathcal I^+ \cap \mathcal I^-=\emptyset$ be index sets of variables that are to be fixed to $1$ or $0$, respectively. Then, we define two sets for \emph{positive and negative fixations} as
$
\D^+ = \{ x\in \{0,1\}^n: x_i = 1 \ \forall i\in\mathcal I^+\}$ and $\D^- = \{ x\in \{0,1\}^n: x_i = 0 \ \forall i\in\mathcal I^-\}$
respectively. In our computational experiments, the number of elements to be fixed is set to the following fraction: 
$$ 
\left\lceil 0.1\cdot\frac{b}{\frac{1}{n}\sum_{i \in \Bar{n}}a_i}\right\rceil.
$$ 
Intuitively, this means $10\%$ of the items in an expected cover solution are fixed to some value.
In each case, the present cover problem was first solved. Then, the favored elements were randomly selected among the items that were not part of the present solution.

We also define a favored solution space $\D^{\ge}$ for given parameters $\alpha\in\{ 0,1\}^n$ and $\beta\in \N$:
\[\D^{\ge} = \{ x\in \{0,1\}^n: \sum_{i\in\mathcal I} \alpha_i x_i \ge \beta\}.\]
For our computational experiments, $10\%$ of all items are randomly selected to be part of the constraint and $\beta=1$, i.e., at least one item from the constraint set must be selected.

For each of the three types of favored solution spaces, we need to be able to take the complement of $\mathcal{D}$ to model the constraint $y\in \mathcal X\setminus \mathcal D$ in the separation problems for the strong CEs. The complement of $\mathcal{D}^\ge$ is given by the knapsack constrained set
\[
\{ x\in \{0,1\}^n: \sum_{i\in\mathcal I} \alpha_i x_i \le \beta - 1\}.
\]
Similarly, the complements of $\mathcal D^+$ and $\mathcal D^-$ are
\[ 
\{ x\in \{0,1\}^n: \sum_{i\in\mathcal I^+} x_i \le |\mathcal I^+| - 1\},\] \[ 
\{ x\in \{0,1\}^n: \sum_{i\in\mathcal I^-} x_i \ge 1\},
\]
respectively.

\paragraph{Mutable parameter space.}
For the mutable parameter spaces, we consider only settings where parameters can be changed homogeneously by a flat percentage of $5\%$. For the distance $\delta$ we use the sum of absolute deviations, i.e., the $\ell_1$-norm, as the objective function. Thus, in total we get $160\times 3\times 2 = 960$ different instances.

\paragraph{Implementation.}
We implemented all proposed algorithms in Python using Gurobi $12.0.0$ as a solver.
All data and code are made available upon request.
All computations were ran on RWTH Aachen Universities' High-Performance Computing (HPC) CLAIX on Intel Xeon 8468 Sapphire nodes with a single $2.1$ GHz core and $5$GB RAM each. We used  a maximum runtime of 10h. A full implementation of all code is provided under \cite{codetopaper}
\iftechreport
    \footnote{See \href{https://zenodo.org/records/18110317}{https://zenodo.org/records/18110317}}.
\fi

\subsection{Computational Results on Knapsack Instances}

\paragraph{Existence and cost of counterfactual explanations. } In total, we found an optimal counterfactual for $668$ instances among $960$. Of the remainder, in $10$ instances of size $30$ and in $32$ instances of size $40$ the runtime of $10$h was insufficient to prove optimality or infeasibility. A further $250$ instances were flagged as infeasible, generally due to finding an infeasible lower bound subproblem.
Of the $250$ infeasible instances, only four were strongly correlated, and the remainder were among the uncorrelated instances. 

Finding a feasible solution required on average a relative change of $0.25\pm0.48\%$ in weights, with a range of  $0\%-3\%$. However, note that for many uncorrelated instances, we did not find a feasible solution. Thus, for those instances, a larger mutable parameter space might be necessary.
Note that larger parameter spaces imply larger ranges of $\mathcal C$ as well. 
In our study, up to $2590$ different potential objective values were considered per instance with an average of $1160\pm 597$ values.
This excludes the effect of the lower bound proposed in Lemma \ref{lem:lower_bound} shown in the next subsection.

\paragraph{Effect of lower bound.} We find that the lower bound reduces the number of iterations required to solve an instance to optimality by more than $50\%$.
Figures \ref{fig:pd_integral_strong} and \ref{fig:pd_integral_weak} illustrate this for strong and weak CEs, respectively.

\begin{figure}[h!]
    \centering
    \includegraphics[width=0.7\linewidth]{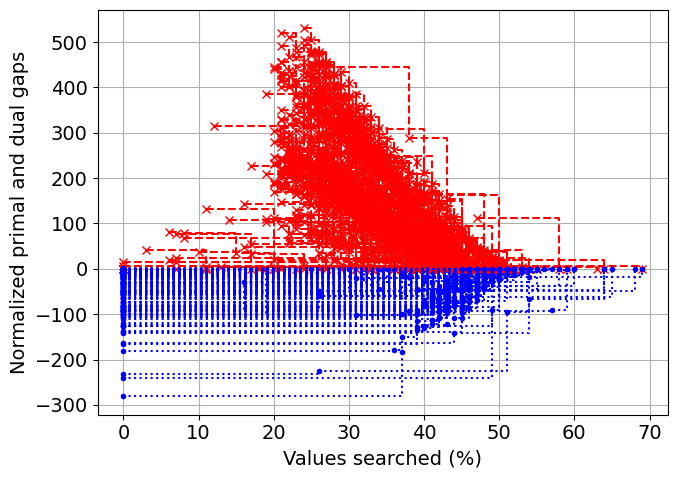}
    \caption{Progress towards optimal solution for $480$ strong CE instances of size $10,20,30, 40$ that were solved to optimality. \textcolor{red}{Red} are \textcolor{red}{primal} and \textcolor{blue}{blue} are \textcolor{blue}{dual} bounds. The $x$-axis is normalized to the total number of values in $\mathcal C$. The $y$-axis is normalized around each instance's optimum.}
    \label{fig:pd_integral_strong}
\end{figure}

\begin{figure}[h!]
    \centering
    \includegraphics[width=0.7\linewidth]{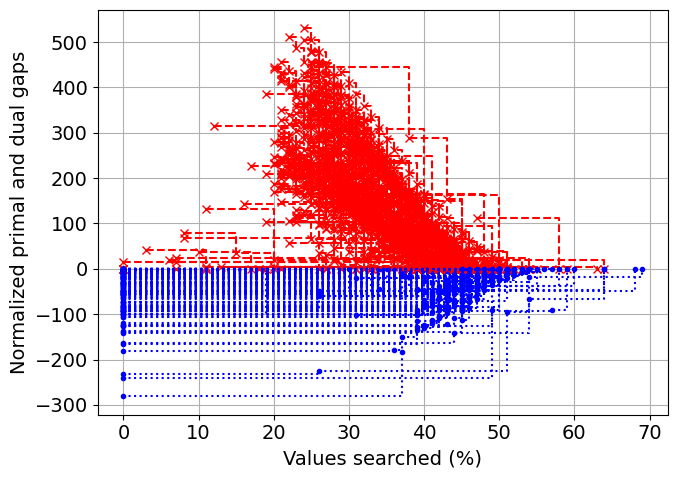}
    \caption{Progress towards solution for $480$ weak CE  instances of size $10,20,30, 40$ that were solved to optimality. \textcolor{red}{Red} are \textcolor{red}{primal} and \textcolor{blue}{blue} are \textcolor{blue}{dual} bounds. The $x$-axis is normalized to the total number of values in $\mathcal C$. The $y$-axis is normalized around each instance's optimum.}
    \label{fig:pd_integral_weak}
\end{figure}

We also observe that good solutions are centered around the halfway point of $50\%$, which tends to be close to the original objective value.
Thus heuristically searching for small changes to the weights that lead to solutions close to the original objective function value appears to be a promising search strategy. However, this sacrifices the lower bound from Lemma \ref{lem:lower_bound} that significantly accelerates the exact approach.
The observations above apply for both weak and strong CEs equally.

\paragraph{Scaling of runtime.} As to be expected, runtime increases significantly with instance size. Figures
\ref{fig:runtimes_strong} and \ref{fig:runtimes_weak} illustrate this. The number of cuts required scales equally with instance size, as can be seen in Figures \ref{fig:cuts_strong} and \ref{fig:cuts_weak}. In summary, CE can be computed using the algorithms provided. However, runtime and memory demands scale exponentially with instance size, which is to be expected for a $\Sigma_2^p$-complete problem. 

\begin{figure}[h!]
    \centering

    \begin{subfigure}{0.48\linewidth}
        \centering
        \includegraphics[width=\linewidth]{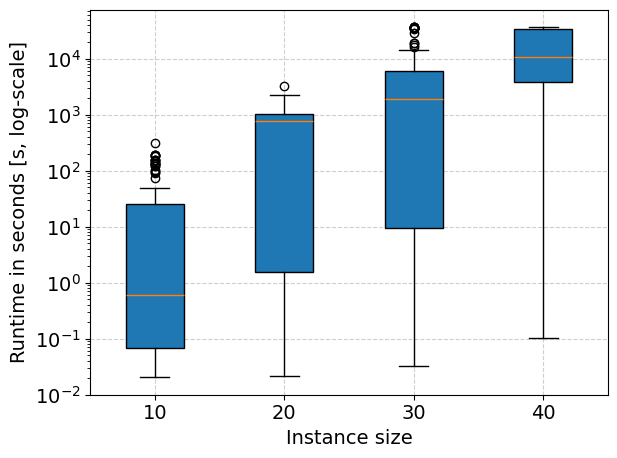}
        \caption{Strong CE runtimes}
        \label{fig:runtimes_strong}
    \end{subfigure}
    \hfill
    \begin{subfigure}{0.48\linewidth}
        \centering
        \includegraphics[width=\linewidth]{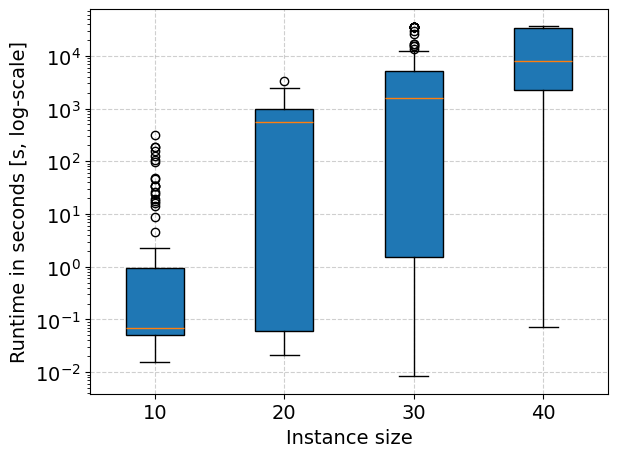}
        \caption{Weak CE runtimes}
        \label{fig:runtimes_weak}
    \end{subfigure}

    \vspace{0.5em}

    \begin{subfigure}{0.48\linewidth}
        \centering
        \includegraphics[width=\linewidth]{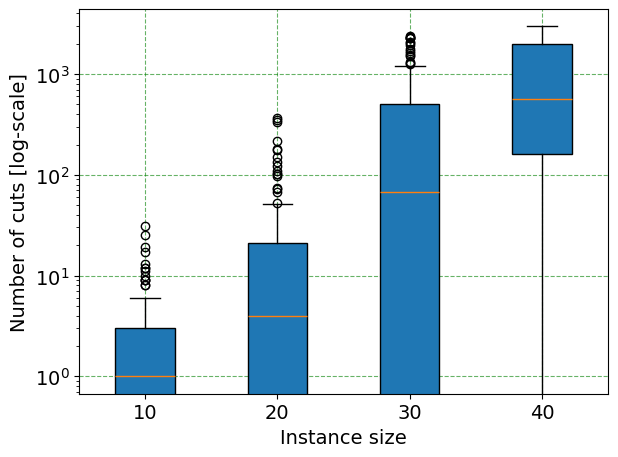}
        \caption{Strong CE number of cuts}
        \label{fig:cuts_strong}
    \end{subfigure}
    \hfill
    \begin{subfigure}{0.48\linewidth}
        \centering
        \includegraphics[width=\linewidth]{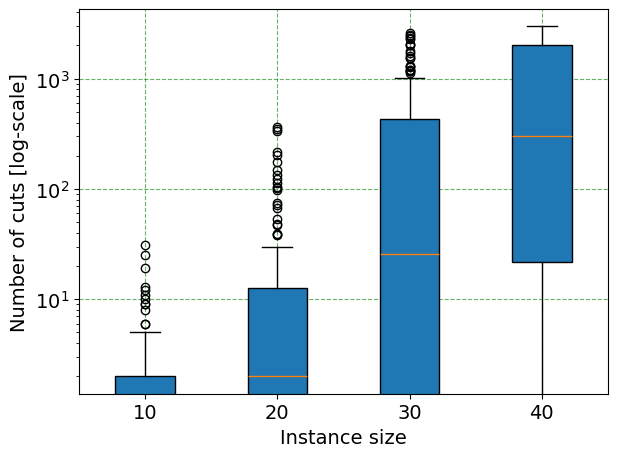}
        \caption{Weak CE number of cuts}
        \label{fig:cuts_weak}
    \end{subfigure}

    \caption{Boxplot comparison of runtimes (top row) and total number of generated cuts (bottom row) for $480$ strong and weak CE instances with sizes $10$, $20$, $30$, and $40$. Instances not solved to optimality are assigned $36000$s ($10$h).}
    \label{fig:runtime_cut_comparison}
\end{figure}

We can also illustrate this using primal-dual gaps. For that we restrict the data displayed to the 1st and 2nd item from each size, i.e $10$, $20$, $30$, and $40$. This is shown in Figure \ref{fig:pd_integral_weak_sample}.

\begin{figure}[h!]
    \centering
    \includegraphics[width=0.7\linewidth]{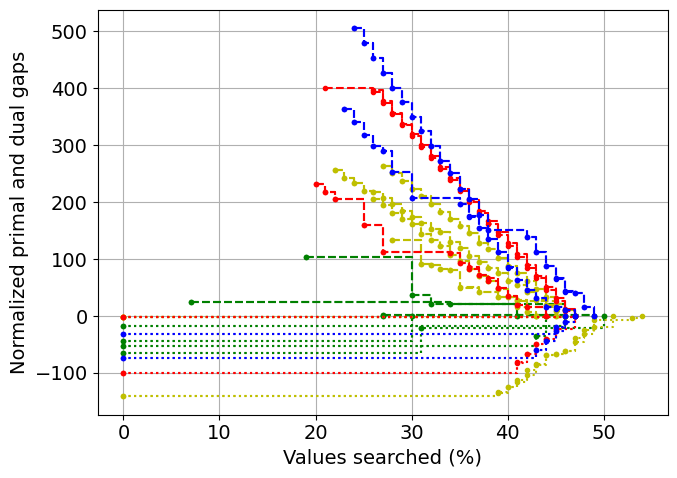}
    \caption{Progress towards solution for $64/96$ strong CE instances solved to optimality. \textcolor{blue}{Blue} instances are size $40$, \textcolor{red}{red} instances are size $30$, \textcolor{yellow}{yellow} and \textcolor{green}{green} correspond to sizes $20, 10$. The $x$-axis is normalized to the total number of values in $\mathcal C$. The $y$-axis is normalized around each instance's optimum.}
    \label{fig:pd_integral_weak_sample}
\end{figure}

\paragraph{Deriving insights from counterfactuals.} Computing CEs can also enable us to quantify (partial) decision values in terms of their contribution to an objective.
Although strongly correlated instances are generally considered to be harder, the fact that all items have a similar cost/weight ratio means small changes in weights are sufficient to add/remove an element from a solution.

As an example, consider CEs for the cover instance in Table \ref{tab:strongly_correlated_instance}.

\begin{table}[h!]
    \centering
    \caption{Exemplary strongly correlated cover instance. For a demand/capacity of $2358$, an optimal cover picks items $1$, $2$, and $7$ for a total objective of $2701$.}
    \begin{tabular}{c||c|c|c|c|c|c|c|c|c|c}
         Item & 0 & 1 & 2 & 3 & 4 & 5 & 6 & 7 & 8 & 9 \\ \hline\hline
         Weight & 135 & 848 & 764& 256& 496& 450& 652 & 789 &  94 & 29 \\
         Costs & 235 &  948 &  864 &  356 &  596 &  550 &  752 &  889 &  194 &  129
    \end{tabular}\label{tab:strongly_correlated_instance}
\end{table}

We compute a strong CE for every item that is not currently part of the optimal solution, with a positive fixation for each. More precisely, for each item $i$ that is not part of the optimal solution we calculate the minimal parameter change needed such that item $i$ is included in all optimal solutions. The results are given in Table \ref{tab:strong_ce_comparison}.

\begin{table}[h!]
    \centering
    \caption{Comparison of CEs for different items in a strongly correlated instance sorted by cost. The cost/weight ratio is based on the initial weights, items $[1, 2, 7]$ have a ratio of $0.89 / 0.88 / 0.88$ respectively.}
    \begin{tabular}{c||c|c|l|c|}
        Item & CE cost & $\frac{cost}{weight}$ ratio &Changes to weights & New solution \\ \hline\hline
        0 & 44 & 0.57 & 1:-9, 2:-5, 7:-30 & [0, 1, 2, 6] \\
        4 & 44 & 0.83 & 1:-41, 2:-3 & [2, 4, 5, 6]\\
        5 & 44 & 0.81 & 1:-42, 7:-2 & [2, 4, 5, 6]\\
        6 & 44 & 0.87 & 1:-42, 7:-2 & [2, 4, 5, 6]\\
        8 & 51 & 0.48 & 2:-5, 4:-2, 6:+1, 7:-39, 8:+4 & [1, 2, 6, 8] \\
        3 & 64 & 0.72 & 2:-38, 5:+21, 7:-6 & [1, 3, 5, 7]\\
        9 & 65 & 0.22 & 2:+42, 3:+22, 9:+1 & [1, 2, 6, 9]\\
    \end{tabular}
    \label{tab:strong_ce_comparison}
\end{table}

We observe different types of modifications: To enforce item $0$, the incumbent solution is made less attractive, thus item $7$ is then substituted by items $0$ and $6$.
For items $4,5$ and $6$, we see that CEs themselves contain alternatives, e.g., to ensure either item is part of an optimal solution multiple equivalent changes of weights lead to the new solution of $[2, 4, 5, 6]$.
Item $8$ is significantly less attractive, enforcing it to be part of a solution requires a higher CE cost, even more so for items $3$ and $9$.
Finally, we observe that simple metrics such as cost/weight are insufficient to capture modeling dynamics -- item $9$ is very good in terms of cost/weight ratio, but still unattractive for optimal solutions. Thus, CEs allow for more differentiated evaluation of (partial) solutions.

\subsection{An Application to the Resource Constrained Shortest Path Problem}
In many cases, even considering few mutable constraints allows us to derive insights into an instance. Next, we examine the following setting: The task is to route traffic through a given graph, where using an arc comes with a toll. Given a limited toll budget this problem can be modeled as a \emph{Resource Constrained Shortest Path (RSCP)}. A CE in the toll parameters asks now how we have to change the toll such that an optimal path would use a certain edge or not.

\begin{figure}[h!]
    \centering
    \includegraphics[width=0.7\linewidth]{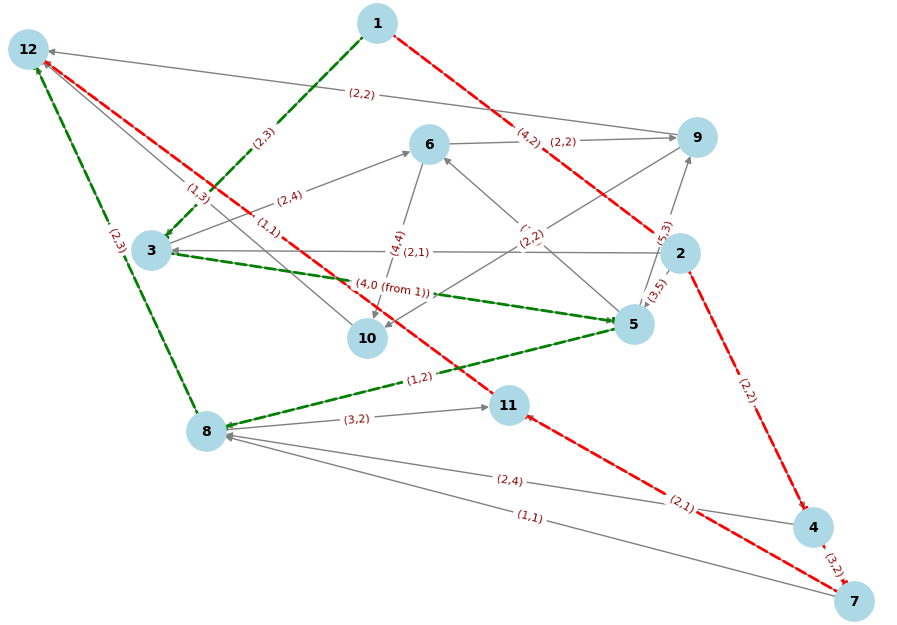}
    \caption{CE for RCSP problem. Edge values indicate (length, toll). The \textcolor{red}{red} path is the original shortest path for a toll budget of $8$; the least change in toll that leads to $e_{(1,2)}$ not being part of an optimal solution is given by lowering the toll of $e_{(3,4)}$ by one unit, resulting in the \textcolor{green}{green} RCSP.}
    \label{fig:rcsp}
\end{figure}

Figure \ref{fig:rcsp} illustrates a counterfactual in such a setting for a toy example. 
If we assume a toll budget of $8$ units, initially, the red path is the shortest path, with total length $12$. 
Assume the arc $e_{(1,2)}$ goes through the city center and we would like to avoid routing traffic through there whenever possible.
Our question would then be: What is the smallest change so that the arc $e_{(1,2)}$ is no longer part of a shortest path? 
The result is given by the green path, which has total length $9$ and requires a single change to the data, i.e., lowering the toll of arc $e_{(3,5)}$ by one unit. 
In our example, this would mean that if we wanted to avoid routing traffic through the city center $e_{(1,2)}$, the best way to do so is to waive the toll on $e_{(3,5)}$ making it more attractive to bypass the city center.

In terms of mathematics, note that the RCSP contains multiple flow conservation constraints. However, since those are immutable, they do not affect our algorithm. Furthermore, we can bound the cost values $c$  by finding the optimal RCSP for each feasible value of $b$, i.e., the toll budget.

\paragraph{Solving RCSP pricing problems.} 
Finally, we test our algorithm on a set of benchmark instances from literature. For that, we use the SPPRCLIB data set provided by Jepsen et al.~\citep{jepsen2008branch}.
The instances represent pricing problems from solving capacitated vehicle routing problems with column generation, see~\citep{Fukasawa2006}. As such, they contain negative node weights, as well as a resource constraint on the nodes.
In Appendix~\ref{app:rcsp_ilp}, we provide a ILP formulation of such problems. We restrict ourselves to the case that source and sink are the first and last indexed node, respectively, and round node weights to the nearest thousand and scaled down accordingly, e.g. $43 000$ to $43$. The instances are complete graphs with $45-262$ nodes.
We investigate what would be the least change to the weights in the resource constraint by a maximum of $10\%$ to exclude the current first node from the shortest path. In the context of a pricing problem, this would imply computing a least change to the duals that would have led to a different column being added.  

Of $44$ instances, we found a weak CE for $22$. Of these $22$ instances, for $5$ the original solution was already a counterfactual explanation. Finding a CE took an average of $15\pm19$s and a maximum of $63$s. In about half the cases, the new RSCPs only have small alterations, for the other half, a completely new path is chosen. 
We note that we can thus compute CEs also for larger and more complex problems, provided the mutable part is limited. In the context of RSCP, this might allow insights into reoccurring patterns or the lack thereof in column generation. 

\section{Conclusion}

In this work, we studied weak and strong counterfactual explanations (CEs) for integer (linear) optimization problems. We showed that in the general case, computing a counterfactual increases complexity by one level in the polynomial hierarchy, that is, $\NP$-complete problems have a counterfactual that is $\Sigma_2^p$-complete, which is equivalent to general bilevel integer linear programming. We present algorithms to calculate counterfactual explanations for integer linear problems and implemented a tractable version thereof, which we evaluated on cover problem instances. We found that the proposed approach works well for instances with up to $40$ items.

There are several open research directions to extend our work. Further research could consider and exploit specific (combinatorial) structures of the present problem to detect situations where calculating CEs is on a lower rank of the polynomial hierarchy and computationally less demanding. Developing heuristics also presents another promising avenue for further research. These methods can deliver computationally efficient solutions that closely approximate the optimal results. Beyond their standalone utility, these heuristic solutions can also serve as effective starting points for the exact algorithms that we proposed here.


\section{Code and Data Disclosure}\label{sec:Code and Data Disclosure} The code and data to support the numerical experiments in this paper can be found at \cite{codetopaper}
\iftechreport
    , respectively under the link \href{https://zenodo.org/records/18110317}{https://zenodo.org/records/18110317}.
\fi

\iftechreport
\section*{Acknowledgements}
Computations were performed with computing resources granted by RWTH Aachen University High Performance computing.
We thank Yosuke Onoue, who build and hosts the \href{https://github.com/likr/kplib}{\texttt{klib}} knapsack instance library.
We thank Christina Büsing and Hector Geffner for enabling this research collaboration, as well as Christoph Grüne for insightful discussions on complexity theory.

\bibliographystyle{alpha} 
\bibliography{references}
\fi

\clearpage

\iftechreport
\appendix
\else
\begin{APPENDICES}
\fi

\section*{Appendix}

\section{Algorithms for Strong CEs}
\begin{algorithm}
\caption{Solve Problem \eqref{eq:master_problem_strongCE}}\label{alg:solve_master_problem_strong}
\begin{algorithmic}
\Require $v, \hat a,\hat b,\hat c,\D, \X, \H_{a,b}$, known upper bound $d^*$ for \eqref{eq:weak_CE_problem}
\Ensure Optimal solution $(a^*(v),b^*(v))$ of \eqref{eq:master_problem}.
\State set $\mathcal Y \leftarrow \emptyset$
\State optimal$\leftarrow$ false
\While{not optimal}
\State Calculate an optimal solution $(\tilde a, \tilde b)$ of problem
\begin{align*}
    \min_{a,b,x} \ & \delta( a,\hat a) + \delta( b,\hat b ) \\
    s.t. \quad & a^\top y \le b - 1 \quad \forall y\in \mathcal Y, \\
    &\hat c^\top x = v, \\
    & a^\top x \ge b, \\
    & x\in\D\cap \X, \\
    & \delta(a,\hat a) + \delta (b,\hat b) \le d^*, \\
    & (a,b)\in \H_{a,b} .
\end{align*}
\State If problem is infeasible, stop and return \textit{infeasible}.
\State Otherwise, calculate an optimal solution $\tilde y$ of the separation problem
\begin{align*}
\max \ & \tilde a^\top y \\
s.t. \quad & \hat c^\top y \le v, \\ 
& y\in\X\setminus \D.
\end{align*}
\If{$\tilde a^\top \tilde y \ge \tilde b$}
\State $\mathcal Y\leftarrow \mathcal Y\cup \{ \tilde y\}$
\Else
\State optimal $\leftarrow$ true \Comment{End While Loop}
\EndIf
\EndWhile
 \ \\
\Return $(\tilde a, \tilde b)$
\end{algorithmic}
\end{algorithm}

\begin{algorithm}
\caption{Optimal Strong CE for Mutable Constraint Parameters.}\label{alg:strong_only_constraint_mutable}
\begin{algorithmic}
\Require $\hat a,\hat b,\hat c,\D, \H_{a,b}$
\Ensure Optimal weak CE $(a^*,b^*)$.
\State Best known solution: $(a^*,b^*) = \emptyset$
\State Best known distance: $d^* = \infty$
\State Best known lower bound: $lb = -\infty$
\State Calculate bounds $\underline c$, $\bar c$ as Lemma \ref{lem:bounds_cmin_cmax}.
\ForAll{$v=\underline c, \underline c + 1,  \ldots , \bar c$}
\State Set $lb$ to the optimal value of \eqref{eq:lower_bound_strongCE} with $\bar v = v$.
\If{$lb \ge d^*$}
\State Stop and \textbf{return} $(a^*,b^*)$.
\Else
\State Calculate opt. solution $(a^*(v),b^*(v))$ of Problem \eqref{eq:master_problem_strongCE} by Algorithm \ref{alg:solve_master_problem_strong}. 
\If{$\delta(a^*(v),\hat a) + \delta(b^*(v),\hat b) < d^*$}
\State $(a^*,b^*) \leftarrow (a^*(v),b^*(v))$
\State $d^*\leftarrow \delta(a^*(v),\hat a) + \delta(b^*(v),\hat b)$
\EndIf
\EndIf
\EndFor \\
\Return $(a^*,b^*)$
\end{algorithmic}
\end{algorithm}

\begin{algorithm}
\caption{Solve Problem \eqref{eq:strong_CE_all_mutable_big-M_reformulation}}\label{alg:solve_strong_CE_all_parameters}
\begin{algorithmic}
\Require $\hat a,\hat b,\hat c,\D, \X, \H$
\Ensure Optimal solution $(a^*,b^*,c^*)$ of \eqref{eq:strong_CE_all_mutable_big-M_reformulation}.
\State set $\mathcal Y \leftarrow \emptyset$
\State optimal$\leftarrow$ false
\While{not optimal}
\State Calculate an optimal solution $(\tilde a, \tilde b, \tilde c, \tilde x)$ of problem
    \begin{align*}
    \min_{c,a,b,x,z} \ & \delta(c,\hat c) + \delta(a,\hat a) + \delta(b,\hat b) \\
    s.t. \quad & c^\top x \le c^\top y - 1 + M z_y \quad \forall y\in \mathcal Y, \\
    & a^\top y \ge b - Mz_y \quad \forall y\in \mathcal Y, \\
    &a^\top y \le b - 1 + M(1-z_y) \quad \forall y\in \mathcal Y, \\
    & x\in \D\cap \X, \\
    & a^\top x \ge b, \\
    & z_y\in \{ 0,1\} \quad \forall y\in \mathcal Y,\\
    & (c,a,b)\in \H.
\end{align*}
\State Calculate an optimal solution $\tilde y$ of the separation problem
\begin{align*}
\min \ &\tilde c^\top y \\
s.t. \quad &\tilde a^\top y \ge \tilde b, \\
& y\in\X\setminus \D.
\end{align*}
\If{$\tilde c^\top \tilde y < \tilde c^\top \tilde x$}
\State $\mathcal Y\leftarrow \mathcal Y\cup \{ \tilde y\}$
\Else
\State optimal $\leftarrow$ true \Comment{End While Loop}
\EndIf
\EndWhile
 \ \\
\Return $(\tilde a, \tilde b, \tilde c)$
\end{algorithmic}
\end{algorithm}

\section{Integer Programming Formulations for the Resource Constrained Shortest Path Problems}\label{app:rcsp_ilp}
For the toy example, let $c_{ij}$ be costs for arc $ij \in A$ and $w_{ij}$ the respective weights. Furthermore, we designate a start node $s$ and a target node $t$.
Then, we define a decision variable $x_{ij} \in \{0,1\}$ that is $1$ if arc $ij$ is part of the shortest path and $0$ otherwise. The RCSPP is then given by:

\begin{align}
\text{min} \quad 
  & \sum_{(i,j)\in A} c_{ij} x_{ij} \notag\\\notag
\text{s. t.} \quad 
  & \sum_{j:(i,j)\in A} x_{ij} - \sum_{j:(j,i)\in A} x_{ji} =
    \begin{cases}
      1, & i = s, \\
     -1, & i = t, \\
      0, & \text{else,}
    \end{cases}
    && \forall i \in V, \\
  & \sum_{(i,j)\in A} w_{ij} \, x_{ij} \le b, \tag{R}
    &&  \\
  & x_{ij} \in \{0,1\}, 
    && \forall (i,j)\in A.\notag
\end{align}
In our setting, we assume that only the resource constraint $R$ is mutable. 

For the SPPRCLIB instances, the resource constraints are formulated on nodes, not edges. Furthermore, nodes also incur cost $c_k$ for a node $k$.
Thus, we use a modified formulation, as seen below:

\begin{align}
\text{min} \quad 
  & \sum_{(i,j)\in A} c_{ij} x_{ij} + \sum_{(i,j) \in A} \frac{c_i+c_j}{2}x_{ij} + \frac{c_{s} + c_{t}}{2}\notag\\\notag
\text{s. t.} \quad 
  & \sum_{j:(i,j)\in A} x_{ij} - \sum_{j:(j,i)\in A} x_{ji} =
    \begin{cases}
      1, & i = s, \\
     -1, & i = t, \\
      0, & \text{else,}
    \end{cases}
    && \forall i \in V, \\
  & \sum_{(i,j) \in A} \frac{w_i+c_j}{2}x_{ij} + \frac{w_{s} + w_{t}}{2} \le b, \tag{R}
    &&  \\
  & x_{ij} \in \{0,1\}, 
    && \forall (i,j)\in A.\notag
\end{align}

\iftechreport
\else
\end{APPENDICES}

\ACKNOWLEDGMENT{Computations were performed with computing resources granted by RWTH Aachen University High Performance computing.
We thank Yosuke Onoue, who build and hosts the \href{https://github.com/likr/kplib}{\texttt{klib}} knapsack instance library.
We thank Christina Büsing and Hector Geffner for enabling this research collaboration, as well as Christoph Grüne for insightful discussions on complexity theory.}

\bibliographystyle{informs2014} 
\bibliography{references}
\fi

\end{document}